\documentclass[11pt,reqno,tbtags]{amsart}
\usepackage{amssymb}
\usepackage{natbib}
\bibpunct[, ]{[}{]}{;}{n}{,}{,}

\numberwithin{equation}{section}

\allowdisplaybreaks

\newtheorem{theorem}{Theorem}[section]
\newtheorem{lemma}[theorem]{Lemma}

\newtheorem{corollary}[theorem]{Corollary}

\theoremstyle{definition}
\newtheorem{example}[theorem]{Example}

\newtheorem{remark}[theorem]{Remark}

\newtheorem*{acks}{Acknowledgements}

\theoremstyle{remark}

\newenvironment{romenumerate}[1][0pt]{% optional argument changes indentation
\addtolength{\leftmargini}{#1}\begin{enumerate}% gives (i), (ii) etc.
 }{\end{enumerate}}

\newcounter{oldenumi}
{\setcounter{oldenumi}{\value{enumi}}
\begin{romenumerate} \setcounter{enumi}{\value{oldenumi}}}
{\end{romenumerate}}

\newcounter{thmenumerate}
\newenvironment{thmenumerate}
{\setcounter{thmenumerate}{0}%
 \def\item{\par% \ifnum\thethmenumerate=0\else\par\fi %\noindent\fi
 \refstepcounter{thmenumerate}\textup{(\roman{thmenumerate})\enspace}}
}
{}

\newcounter{xenumerate}   %no left indentation; thus wider lines

\newcommand\pfitem[1]{\par(#1):}

\newcounter{step}

\newcommand{\refT}[1]{Theorem~\ref{#1}}
\newcommand{\refC}[1]{Corollary~\ref{#1}}
\newcommand{\refL}[1]{Lemma~\ref{#1}}
\newcommand{\refR}[1]{Remark~\ref{#1}}
\newcommand{\refS}[1]{Section~\ref{#1}}
\newcommand{\refSS}[1]{Section~\ref{#1}}% subsection?
\newcommand{\refSSS}[1]{Subsection~\ref{#1}}

\newcommand{\refE}[1]{Example~\ref{#1}}

\newcommand{\refand}[2]{\ref{#1} and~\ref{#2}}

\begingroup
  \divide\count255 by 60
  \multiply\count255 by -60
  \advance\count255 by \time
  \ifnum \count255 < 10 \xdef\klockan{\the\count1.0\the\count255}
  \else\xdef\klockan{\the\count1.\the\count255}\fi
\endgroup

\newcommand{\sumj}{\sum_{j=0}^\infty}
\newcommand{\sumji}{\sum_{j=1}^\infty}
\newcommand{\sumk}{\sum_{k=1}^\infty}
\newcommand{\suml}{\sum_{\ell=0}^\infty}

\newcommand{\sumko}{\sum_{1\le k<\infty}}
\newcommand{\sumkoo}{\sum_{1\le k\le\infty}}

\newcommand\set[1]{\ensuremath{\{#1\}}}
\newcommand\bigset[1]{\ensuremath{\bigl\{#1\bigr\}}}

\newcommand\xpar[1]{(#1)}
\newcommand\bigpar[1]{\bigl(#1\bigr)}
\newcommand\Bigpar[1]{\Bigl(#1\Bigr)}

\newcommand\lrpar[1]{\left(#1\right)}
\newcommand\bigsqpar[1]{\bigl[#1\bigr]}

\def\rompar(#1){\textup(#1\textup)}    % usage: \rompar(...)

\newcommand\Bigparfrac[2]{\Bigpar{\frac{#1}{#2}}}

\def\xexp(#1){e^{#1}}

\newcommand\floor[1]{\lfloor#1\rfloor}

\newcommand\ntoo{\ensuremath{{n\to\infty}}}

\newcommand\bmin{\wedge}
\newcommand\bmax{\vee}

\newcommand\downto{\searrow}
\newcommand\upto{\nearrow}

\newcommand\iid{i.i.d.\spacefactor=1000}    
\newcommand\ie{i.e.\spacefactor=1000}
\newcommand\eg{e.g.\spacefactor=1000}

\newcommand\cf{cf.\spacefactor=1000}
\newcommand{\as}{a.s.\spacefactor=1000}
\newcommand{\aex}{a.e.\spacefactor=1000}
\newcommand\whp{whp}

\newcommand\ii{\mathrm{i}}

\newcommand{\tend}{\longrightarrow}

\newcommand\pto{\overset{\mathrm{p}}{\tend}}

\newcommand\lto{\overset{\mathrm{L^1}}{\tend}}
\newcommand\eqd{\overset{\mathrm{d}}{=}}

\newcommand\bbR{\mathbb R}

\newcommand\bbN{\mathbb N}  %{1,2,...}

\renewcommand\Re{\operatorname{Re}}

\newcommand\E{\operatorname{\mathbb E{}}}
\renewcommand\P{\operatorname{\mathbb P{}}}

\newcommand\Po{\operatorname{Po}}

\newcommand\ga{\alpha}

\newcommand\gd{\delta}

\newcommand\gam{\gamma}

\newcommand\gk{\kappa}
\newcommand\gl{\lambda}

\newcommand\eps{\varepsilon}

\renewcommand\phi{\varphi}

\newcommand\cA{\mathcal A}
\newcommand\cB{\mathcal B}
\newcommand\cC{\mathcal C}

\newcommand\cE{\mathcal E}

\newcommand\cI{\mathcal I}

\newcommand\cT{{\mathcal T}}

\newcommand\cV{\mathcal V}

\newcommand\ett[1]{\boldsymbol1[#1]}

\newcommand\qq{^{1/2}}

\newcommand\qqw{^{-1/2}}

\newcommand\qw{^{-1}}
\newcommand\qww{^{-2}}

\renewcommand{\=}{:=}

\newcommand\intoi{\int_0^1}

\newcommand\ooo{[0,\infty)}
\newcommand\ooop{(0,\infty)}

\newcommand\dd{\,\mathrm{d}}

\newcommand\rhs{right-hand side}

\newcommand\gnp{\ensuremath{G(n,p)}}
\newcommand\gnm{\ensuremath{G(n,m)}}

\newcommand\xs{\ensuremath{{\mathbf x}_n}}

\newcommand\xss{\ensuremath{(\xs)_{n\ge 1}}}
\newcommand\nv[1]{v_{#1}}
\newcommand\nvn{\nv{n}}
\newcommand\nun{\nu_n}
\newcommand\nuni{\nu_n^1}

\newcommand\vxs{\ensuremath{\mathcal V}}% vertex space
\newcommand\pij{p_{ij}}
\newcommand\gxxx[2]{\ensuremath{G^\vxs(#1,#2)}}
\newcommand\gnx[1]{\ensuremath{G(n,#1)}}
\newcommand\gnxx[1]{\ensuremath{G^\vxs(n,#1)}}

\newcommand\gnk{\gnx{\kk}}
\newcommand\gnkx{\gnxx{\kk}}

\newcommand\gnkxn{\gnxx{\kk_n}}
\newcommand\gnln{\gnx{\gl/n}}

\newcommand\norm[1]{\ensuremath{\|#1\|}}
\newcommand\normpx[2]{\ensuremath{\|#1\|_{#2}}}

\newcommand\normll[1]{\normpx{#1}{2}}
\newcommand\normoo[1]{\normpx{#1}{\infty}}
\newcommand\ninf[1]{\normoo{#1}}
\newcommand\on[1]{\normpx{#1}{1}}

\newcommand\sus{\chi}
\newcommand\susq{\widehat\chi}
\newcommand\cc[1]{\mathcal C_{#1}}
\newcommand\cci{\cc{i}}

\newcommand\ccc[1]{|\cc{#1}|}
\newcommand\ccci{\ccc{i}}

\newcommand\sumiK{\sum_{i=1}^K}
\newcommand\sumixK{\sum_{i=2}^K}
\newcommand\kk{\kappa}
\newcommand\kkn{\kappa_n}

\newcommand\ka{\kappa}

\newcommand\bpx[1]{\ensuremath{\mathfrak X_{#1}}}
\newcommand\bpk{\ensuremath{\mathfrak X_{\kk}}}
\newcommand\bpkq{\ensuremath{\widehat{\mathfrak X}_{\kk}}}
\newcommand\qbpk{|\bpk|}
\newcommand\bpxx[1]{\ensuremath{\mathfrak X_{#1}(x)}}
\newcommand\bplk{\ensuremath{\mathfrak X_{\glk}}}
\newcommand\qbplk{|\bplk|}
\newcommand\bpglck{\ensuremath{\mathfrak X_{\glc\kk}}}
\newcommand\nx[1]{\ensuremath{N_{#1}}}
\newcommand\nk{\nx{k}}
\newcommand\nj{\nx{j}}
\newcommand\ngek{\nx{\ge k}}
\newcommand\ngeK{\nx{\ge K}}
\newcommand\susk{\sus(\kk)}
\newcommand\susqk{\susq(\kk)}
\newcommand\rhox[1]{\ensuremath{\rho_{#1}}}
\newcommand\rhok{\rhox{k}}
\newcommand\rhoj{\rhox{j}}
\newcommand\rhogek{\rhox{\ge k}}
\newcommand\rhogeK{\rhox{\ge K}}
\newcommand\rhokk{\rhox{\kk}}
\newcommand\rhoka{\rho(\kk)}
\newcommand\rhokkk{\rhox{k}(\kk)}
\newcommand\sss{\mathcal S}
\newcommand\sssq{\sss\times\sss}
\newcommand\sssmu{(\sss,\mu)}
\newcommand\ints{\int_{\sss}}
\newcommand\intss{\int_{\sss^2}}
\newcommand\muq{\widehat\mu}
\newcommand\kkq{\widehat\kk}
\newcommand\kkqq{\widetilde\kk}
\newcommand\kke{\kk_\eps}

\newcommand\kko{\kk_0}

\newcommand\tk{T_{\kk}}
\newcommand\tkq{T_{\kkq}}
\newcommand\Tx[1]{T_{#1}}
\newcommand\Txq[1]{T_{\widehat{#1}}}
\newcommand\Tq[1]{\hat T_{#1}}

\newcommand\tlk{\Tx{\gl\kk}}
\newcommand\tlkq{\Txq{\gl\kk}}
\newcommand\tql{\Tq{\gl}}
\newcommand\tqz{\Tq{z}}
\newcommand\TTx[1]{\widetilde T_{#1}}
\newcommand\ttl{\TTx{\gl}}
\newcommand\tti{\TTx{1}}
\newcommand\innprod[1]{\langle#1\rangle}
\newcommand\biginnprod[1]{\bigl\langle#1\bigr\rangle}
\newcommand\innprodmu[1]{\innprod{#1}_{\mu}}

\newcommand\innprodmuq[1]{\innprod{#1}_{\muq}}
\newcommand\nn{^{(n)}}
\newcommand\musss{\mu(\sss)}
\newcommand\musssqw{\mu(\sss)\qw}

\newcommand\llmu{L^2(\mu)}

\newcommand\qir{irreducible}
\newcommand\op{o_{\mathrm p}}

\newcommand\Thp{\Theta_{\mathrm p}}
\newcommand\glc{\gl_{\mathrm{cr}}}
\newcommand\glo{\gl_{0}}
\newcommand\glk{\gl\kk}
\newcommand\glck{\glc\kk}
\newcommand\extrho{\rho^+}
\newcommand\extrhoz{\extrho_z}
\newcommand\ul{U_\gl}
\newcommand\rholk{\rho_{\glk}}
\newcommand\uqq{(1-\rholk)\qq}
\newcommand\psia{\psi_A}
\newcommand\gla{\gl_A}
\newcommand\oa{O(\norm{A})}
\newcommand\oaa{O(\norm{A}^2)}
\newcommand\ot{O(\norm{T'-T})}
\newcommand\ott{O(\norm{T'-T}^2)}

\newcommand\reps{r_\eps}
\newcommand\ddx{\frac{\dd}{\dd x}}
\newcommand\gi[1]{G_{#1}^{I}}
\newcommand\gii[1]{G_{#1}^{II}}
\newcommand\giii[1]{G_{#1}^{III}}
\newcommand\gin{\gi n}
\newcommand\giin{\gii n}
\newcommand\giiin{\giii n}
\newcommand\glij{\gl_{ij}}
\newcommand\glq{\sqrt\gl}
\newcommand\susqer{\susq_1}
\newcommand\susqd{\susq_\delta}
\newcommand\hx{\hat X}

\newcommand\muqq{\check\mu}
\newcommand\tkqq{T_{\check\ka}}
\newcommand\comp{{\mathrm{c}}}

\newcommand{\Takacs}{Tak\'acs}
\newcommand\ER{Erd\H os--R\'enyi}
\newcommand\HS{Hilbert--Schmidt}

\newcommand\REM[1]{\par{\raggedright\texttt{[#1]}\par\marginal{XXX}}}

\newcommand\ta{{(\tau)}}
\newcommand\dcut{{\delta_{\square}}}% d?
\newcommand\cn[1]{\norm{#1}_{\square}}
\newcommand\tkn{T_{\ka_n}}
\newcommand\tC{{\tilde C}}
\newcommand\tG{{\widetilde G}}
\newcommand\tn{{\tilde n}}

\begin{document}
\title%[]
{Susceptibility in inhomogeneous random graphs}

\author{Svante Janson}
\address{Department of Mathematics, Uppsala University,
 PO Box 480, SE-751 06 Uppsala, Sweden}
\author{Oliver Riordan}
\address{Mathematical Institute, University of Oxford, 
24--29 St Giles', Oxford OX1 3LB, UK}
\date{May 1, 2009}% (compiled \today)}

\subjclass[2000]{05C80, 60C05} 
%%{Primary: <subject>; Secondary: <subject>}

\begin{abstract} 
We study the susceptibility, i.e., the mean size of the component
containing a random vertex, in a general model of inhomogeneous random
graphs. This is one of the fundamental quantities associated to 
(percolation) phase transitions; in practice one of its main uses
is that it often gives a way of determining the critical point
by solving certain linear equations. Here we relate the susceptibility
of suitable random graphs to a quantity associated
to the corresponding branching process, and study both quantities
in various natural examples.
\end{abstract}

\maketitle

\section{Introduction}\label{S:intro}

The \emph{susceptibility} $\sus(G)$ of a (deterministic or
random) graph $G$ 
is defined as the mean size of the component containing a random
vertex:
\begin{equation}\label{sus1}
 \sus(G)=|G|^{-1}\sum_{v\in V(G)}|\cC(v)|,
\end{equation}
where $\cC(v)$ denotes the component of $G$ containing the vertex $v$.
Thus, if $G$ has $n=|G|$ vertices and 
components $\cci=\cci(G)$, $i=1,\dots, K$, 
where $K=K(G)$ is the number of
components, then
\begin{equation}\label{sus}
  \sus(G)\=\sumiK \frac{\ccci}{n}\ccci
=\frac1n\sumiK \ccci^2.
\end{equation}
Later we shall order the components, assuming as usual
that $\ccc1\ge\ccc2\ge\cdots$.

When the graph $G$ is itself 
random, in some contexts (such as percolation)
it is usual to take the expectation
over $G$ as well as over $v$. Here we do \emph{not} do so: when 
$G$ is random, $\sus(G)$ is a random variable.

\begin{remark}
The term susceptibility comes from physics.
(We therefore use the notation $\sus$, which is standard in physics,
although it usually means something else % chromatic number 
in graph theory.)
The connection with the graph version is through 
(\eg) the Ising model for magnetism and the corresponding
random-cluster model, which is a random graph where the susceptibility
\eqref{sus}, or rather its expectation, corresponds to the magnetic
susceptibility.
\end{remark}

The susceptibility has been much studied for certain models in
mathematical physics. Similarly, in percolation theory, which deals
with certain random infinite graphs, the corresponding quantity is
the (mean) size of the open cluster containing a given vertex, and this
has been extensively studied; see \eg{} \citet{BRbook}.
In contrast, not much rigorous work has been done for finite random
graphs. 
Some results for the \ER{} random graphs \gnp{} and \gnm{} can be
regarded as folk theorems  that have been 
known to experts for a long time.
\citet{Durrett} proves that the expectation
$\E\sus(\gnp)=(1-\gl)\qw+O(1/n)$ if $p=\gl/n$ with $\gl<1$ fixed.
The susceptibility of \gnp{} and \gnm{} is studied in detail
by  \citet{SJ218}.
For other graphs, one rigorous treatment is by
\citet{SW}, who study a class of random graph processes
(including the \ER{} graph process)
and use the susceptibility to study the phase transition in them.

The purpose of the present paper is to study $\sus(\gnkx)$ 
for the inhomogeneous random graph $\gnkx$ introduced in \citet{kernels};
this is a rather general model that includes \gnp{} as a special case.
In fact, much of the time we shall consider the more general setting
of~\cite{cutsub}.  
We review the fundamental definitions from \cite{kernels,cutsub} in
\refS{Sprel} below.

We consider asymptotics as $\ntoo$ and use standard notation such as
$\op$, 
see \eg{} \cite{kernels}. 
All unspecified limits are as \ntoo.

\begin{remark}
  We obtain results for \gnp{} as corollaries to our general
  results, but note that these results are not (and cannot be, because
  of the generality of the model \gnkx) as precise as the results
  obtained by \citet{SJ218}.
The proofs in the two papers are quite different;
the proofs in \cite{SJ218} are based on studying the evolution of the
susceptibility for the random graph process obtained by adding random
edges one by one, using methods from stochastic process theory, while
the present paper is based on the standard branching process
approximation of the neighbourhood of a given vertex. It seems likely
that this method too can be used to give more precise results in the
special case of \gnp, but we have not attempted that. (\citet{Durrett}
  uses this method for the expectation $\E\sus(\gnp)$.)
\end{remark}

The definition \eqref{sus} is mainly interesting in the subcritical
case, when all components are rather small. In the supercritical case,
there is typically one giant component that is so large that it
dominates the sum in \eqref{sus}, and thus $\sus(G)\sim\ccc1^2/n$.
In fact, in the supercritical case of \cite[Theorem 3.1]{kernels}, 
$\ccc1=\Thp(n)$ and $\ccc2=\op(n)$,
and thus 
\begin{equation*}
  \sumiK\ccci^2
=\ccc1^2+O\Bigpar{\ccc2\sumixK\ccci}
=\ccc1^2+O\bigpar{\ccc2 n}
=(1+\op(1))\ccc1^2.
\end{equation*}
(See also \cite[Appendix A]{SJ218} for \gnp.)
In this case, it makes sense to exclude the largest component from the
definition; this is in analogy with percolation theory, where one
studies the mean size of the open cluster containing, say, vertex 0,
given that this cluster is finite.
We thus define the {\em modified susceptibility} $\susq(G)$
of a finite graph
$G$ by
\begin{equation}\label{susq}
  \susq(G)\=\frac1n\sumixK \ccci^2.
\end{equation}
Note that we divide by $n$ rather than by $n-|\cc1|$, which would 
also make sense.

In the uniform case, one interpretation of $\susq(G)$ is that it
gives the rate of growth of the giant component above the critical point.
More generally, if we add a single new edge chosen uniformly
at random to a graph $G$, then the probability that $\cC_i$ becomes joined to $\cC_1$ is 
asymptotically $2\ccci\ccc1/n^2$, and when this happens $\ccc1$ increases
by $\ccci$. Thus (under suitable assumptions), the expected
increase in $\ccc1$ is asymptotically $2\ccc1\sum\ccci^2/n=2\ccc1\susq(G)$.

The results in \cite{kernels} on components of \gnkx{} are based on
approximation by a branching process $\bpk$, see \refS{Sprel}.
We define (at least when $\musss=1$, see \refS{Sprel})
\begin{align}
  \susk&\=\E\qbpk\in[0,\infty],
\\
  \susqk&\=\E\bigpar{\qbpk;\qbpk<\infty}\in[0,\infty].
\end{align}
Thus, $\susk=\susqk$ when the survival probability
$\rhoka\=\P(\qbpk=\infty)=0$ (the subcritical or critical case), 
while
$\susk=\infty\ge\susqk$ when $\rhoka>0$ (the supercritical case).

Our main result is that under some extra conditions, the
[modified] susceptibility of $\gnkx$ converges to $\susk$ [$\susqk$],
see \refS{Smain} and in particular Theorems \refand{Tbounded}{Tiid}.

We also study the behaviour of  $\sus(\glk)$ and
$\susq(\glk)$ as functions of the parameter $\gl\in(0,\infty)$, and in
particular the behaviour at the threshold for existence of a giant
component, see \refS{Sthreshold};
 this provides a way to use the susceptibility to find the
threshold for the random graphs treated here. (See, \eg{}, \citet{Durrett} and
\citet{SW} for earlier uses of this method.)

Finally, we consider some explicit examples and counterexamples in \refS{Sex}.

\begin{remark}
  We believe that similar results hold for the `higher order 
  susceptibilities'
\[ 
\sus_m(G)\=\frac1{|G|}\sum_{v\in V(G)}|\cC(v)|^m
=\frac1{|G|}\sum_i \ccci^{m+1}
,
\]
but we have not pursued this. (For \gnp, see \cite{SJ218}.)
\end{remark}

\begin{acks}
Part of this work was carried out during the programme
``Combinatorics and Statistical Mechanics'' at the Isaac Newton
Institute, Cambridge, 2008, where SJ was supported by a Microsoft
fellowship, and part during a visit of both authors 
to the programme ``Discrete Probability''
at Institut Mittag-Leffler, Djursholm, Sweden, 2009.
\end{acks}

\section{Preliminaries}\label{Sprel}

We review the fundamental definitions from \cite{kernels,cutsub}, but refer to
those papers for details, as well as for references to previous work.
In terms of motivation and applications, our main interest is the model $\gnkx$
of~\cite{kernels}, but for the proofs we sometimes need (or can handle) different generality.

\subsection{The random graph models}

In all variations we start with a measure space $(\sss,\mu)$ with $0<\mu(\sss)<\infty$
(usually, but not always, $\mu$ is a probability measure, \ie,
$\mu(\sss)=1$), and a \emph{kernel} on it, \ie, a symmetric non-negative
measurable function $\kk:\sssq\to\ooo$. 
We assume throughout that $\kk$ is integrable: $\int_{\sss^2}\kk(x,y)\dd\mu(x)\dd\mu(y)<\infty$.

\subsubsection{The general inhomogenous model.}\label{sss_gnkx}
To define $\gnkx$, we assume that we are
given, for each $n\ge1$ (or perhaps for $n$ in another suitable index set
$\cI\subseteq(0,\infty)$), 
a random or deterministic finite sequence 
$\xs=(x_1,x_2,\ldots,x_{\nvn})$ of points in $\sss$.
(For simplicity we write $x_i$ instead of $x\nn_i$.)
We denote the triple $(\sss,\mu,\xss)$ by $\vxs$ and define the
random graph
$G_n=\gnkx$ by first sampling $\xs=(x_1,x_2,\ldots,x_{\nvn})$ and
then, given $\xs$, taking the graph with vertex set
$\set{1,\dots,\nvn}$ and random edges, with edge $ij$ 
present with probability $\min(\kk(x_i,x_j)/n,1)$, independently of 
all other edges. (Alternatively, and almost equivalently, see
\cite{kernels} and \cite{SJ212}, we may use the probability
$1-\exp(-\kk(x_i,x_j)/n)$.) 
We interpret $x_i$ as the \emph{type} of vertex $i$, and call
$(\sss,\mu)$ the \emph{type space}.

We need some technical conditions. In \cite{kernels}, we assume that
$\sss$ is a separable metric space and $\mu$ a Borel measure; we
further assume that if $\nun$ is the (random) measure
$n\qw\sum_{i=1}^{\nvn}\gd_{x_i}$, then 
$\nun \pto\mu$ (with weak convergence of measures); 
in this case $\vxs$ is called a \emph{generalized vertex space}.
In the standard special case when $\nvn=n$ and $\mu(\sss)=1$, $\vxs$
is called a \emph{vertex space}.
Furthermore, in \cite{kernels} it is
assumed that the kernel $\kk$ is \emph{graphical} on
$\vxs$, which 
means that $\kk$ is integrable and \aex{} continuous, and that
the expected number of edges is as expected, \ie, that 
$\E e(\gnkx)/n\to\frac12\int_{\sss^2}\kk$.

Many of the results in \cite{kernels} extend to sequences
$\gnkxn$, where $(\kkn)$ is a sequence of kernels on $\vxs$ that 
is  \emph{graphical on $\vxs$ with limit $\kk$}; 
see \cite{kernels} for the definition and note that this includes the
case when all $\kkn=\kk$ for some graphical kernel $\kk$.

As shown in \cite[Section 8.1]{kernels}, 
if $\vxs$ is a generalized vertex space, we may condition on $\xss$, and
may thus assume that the $\xs$ and, in particular, $v_n$ are
deterministic. Replacing the index $n$ by $v_n$, and renormalizing
appropriately (see \refR{Renorm} below), we may reduce to the case of
a vertex space.

\subsubsection{The \iid\ case.}\label{sss_iid}
Another, often simpler, case of the general model
is when $(\sss,\mu)$ is an arbitrary probability space 
and $(x_1,\dots,x_n)$ are $n$ \iid{} points with
distribution $\mu$; in this case $\kk$ can be any integrable kernel. 
This case was unfortunately not treated in
\cite{kernels}, but corresponding results are shown 
for this case (and in greater generality) in \cite{clustering}.
In this case we call $\vxs=(\sss,\mu,\xss)$ an 
\emph{\iid{} vertex space}.
In this case, to unify the notation, a graphical kernel is thus any
integrable kernel. Many results for this case extend to suitable sequences of kernels,
for example assuming that $\on{\kkn-\kk}\to 0$, as then the general 
setting below applies.

\subsubsection{Cut-convergent sequences}\label{sss_cc}
To define the final variant we shall consider, we briefly recall some definitions.
(A variant of) the Frieze--Kannan~\cite{FKquick} {\em cut norm} of an
integrable function $W:\sss^2\to \bbR$ is simply 
\[
 \sup_{\ninf{f},\,\ninf{g}\le 1} \intss f(x)W(x,y)g(y) \dd\mu(x)\dd\mu(y).
\]
Given an integrable kernel $\ka$ and a measure-preserving bijection $\tau:\sss\to\sss$,
let $\ka^\ta$ be the corresponding {\em rearrangement} of $\ka$, defined by
\[ 
 \ka^\ta(x,y) = \ka(\tau(x),\tau(y)).
\]
We write $\ka\sim \ka'$ if $\ka'$ is a rearrangement of $\ka$.
Given two kernels $\ka$, $\ka'$ on $[0,1]$,
the {\em cut metric} of
Borgs, Chayes, Lov\'asz, S\'os and Vesztergombi~\cite{BCLSV:1}
may be defined by
\begin{equation}\label{cutdef}
 \dcut(\ka,\ka') = \inf_{\ka''\sim \ka'} \cn{\ka-\ka''}.
\end{equation}
There is also an alternative definition via couplings, which also
applies to kernels defined on two different probability spaces;
see~\cite{BCLSV:1,BRmetrics}.

Suppose that $A_n=(a_{ij})$ is an $n$-by-$n$ symmetric matrix with non-negative entries;
from now on any matrix denoted $A_n$ is assumed to be of this form.
Then there is a random graph $G_n=G(A_n)$ naturally associated to $A_n$: the vertex
set is $\{1,2,\ldots,n\}$, edges are present independently, and the probability
that $ij$ is an edge is $\min\{a_{ij}/n,1\}$.
Given $A_n$, there is a corresponding kernel $\ka_{A_n}$ on $[0,1]$ with Lebesgue
measure:
divide $[0,1]^2$ into $n^2$ squares of side $1/n$ in the obvious way, and take
the value of $\ka_{A_n}$ on the $(i,j)$th square to be $a_{ij}$.
Identifying $A_n$ and the corresponding kernel, as shown
in~\cite{cutsub}, many of the results of~\cite{kernels}
apply to $G_n=G(A_n)$ whenever $\dcut(A_n,\ka)\to 0$
for some kernel $\ka$ on $[0,1]$ (or, more generally,
on some probability space $\sss$).

If $A_n$ is itself random, then $G(A_n)$ is defined to have the
conditional distribution just described, given $A_n$. Any results
stating that if $\dcut(A_n,\ka)\to 0$ then $G(A_n)$ has some
property
with probability tending to $1$ 
apply also if $(A_n)$ is random with $\dcut(A_n,\ka)\pto 0$.
(One way to see this is to note that
there is a coupling of the distributions
of the $A_n$ in which $\dcut(A_n,\ka)\to 0$ a.s., and
we may then condition on $(A_n)$.)

Moreover, as shown in~\cite[Sections 1.2 and 1.3]{cutsub}, such
results apply to the models described in the previous
subsections, since in each case
the (random) matrices of edge probabilities obtained 
after conditioning on the vertex types
converge in probability to $\ka$ in $\dcut$.

\subsection{The corresponding branching process}

Given an integrable kernel $\ka$ on a measure space $(\sss,\mu)$,
let $\bpk(x)$, $x\in\sss$, be the multi-type Galton--Watson
branching process defined as follows.
We start with a single particle of type $x$ in generation $0$.
A particle in generation $t$
of type $y$ gives rise to children in generation  
$t+1$ whose types form a Poisson process on $\sss$ with intensity
$\ka(y,z)\dd\mu(z)$.
The children of different particles are independent (given the types
of their parents).

If $\mu$ is a probability measure, we also consider
the branching process $\bpk$ defined as above but
starting with a single particle 
whose type has the distribution  $\mu$. 

Let $|\bpk(x)|$ denote the total population of $\bpk(x)$, and let
\begin{align}
  \rhok(\kk;x)&\=\P(|\bpk(x)|=k),
\qquad k=1,2,\dots,\infty,
\end{align}
and
\begin{align}
 \rhok(\kk)&
\=
\ints\rhok(\kk;x)\dd\mu(x),
\qquad k=1,2,\dots,\infty.
\end{align}
Thus, when $\musss=1$, $ \rhok(\kk)$ is the  probability $\P(|\bpk|=k)$.

For convenience we assume that 
\begin{equation}\label{51}
  \ints\kk(x,y)\dd\mu(y)<\infty
\end{equation}
for all $x\in\sss$; this implies that all sets of children are finite a.s.
This is no real restriction, since our assumption that
$\int_{\sss^2}\kk<\infty$ implies that \eqref{51} holds for
\aex{} $x$, and we may impose \eqref{51} by changing $\kk$ on a null
set, which will \as{} not affect $\bpk$.
(Alternatively, we could work without \eqref{51}, adding the
qualifier ``for \aex{} $x$'' at some places below.)

Since \as{} all generations of $\bpk(x)$ are finite,
it follows that $\rhox\infty(\ka;x)$, the probability that the branching
process is infinite, equals the \emph{survival probability} of
$\bpk(x)$, i.e., the probability that all generations are non-empty.
We use the notation 
$\rho(\kk;x)\=\rhox\infty(\kk;x)$;
for typographical reasons we sometimes also write $\rhokk(x)=\rho(\kk;x)$.
Similarly, 
we write $\rhoka\=\rhox\infty(\kk)$;
if $\musss=1$, this is the survival probability of
$\bpk$. 

We are interested in the analogue of the mean cluster size for the branching processes.
For $\bpk(x)$, we define
\begin{align}
 \sus(\ka;x)
&\=\E\bigpar{|\bpk(x)|}
=\sumkoo k\rhok(\kk;x),
\\
  \susq(\ka;x)
&\=\E\bigpar{|\bpk(x)|; {|\bpk(x)|<\infty}}
=\sumko k\rhok(\kk;x);
\end{align}
thus 
$\sus(\kk;x)=\susq(\kk;x)\le\infty$ 
if $\rho(\kk;x)=0$, and 
$\susq(\ka;x)\le\sus(\ka;x)=\infty$ if $\rho(\kk;x)>0$.
Further, let
\begin{align}
 \sus(\ka)
&\=\musssqw\ints\sus(\ka;x)\dd\mu(x)
=\musssqw\sumkoo k\rhokkk,
\label{suskk}
\\
  \susq(\ka)
&\=\musssqw\ints\susq(\ka;x)\dd\mu(x)
=\musssqw\sumko k\rhokkk.
\label{suskkq}
\end{align}
Thus, if $\musss=1$, 
\begin{align}\label{susqkkb}
 \sus(\ka)
&=\E\bigpar{|\bpk|},
\\
  \susq(\ka)
&=\E\bigpar{|\bpk|; {|\bpk|<\infty}}.
\label{suskkb}
\end{align}

\begin{remark}\label{Renorm}
For a generalized vertex space, where $\mu(\sss)$ may differ from 1, we
may renormalize by replacing 
$\mu$ and $\kk$ by
\begin{align}
  \label{mu'}
\mu'\=\mu(\sss)\qw\mu
\qquad\text{and}\qquad
\kk'\=\mu(\sss)\kk. 
\end{align}
This will not affect $\bpk(x)$, and thus not $\sus(\kk;x)$ and $\susq(\kk;x)$;
further, because of our choice of normalization in \eqref{suskk} and
\eqref{suskkq}, $\sus(\kk)$ and $\susq(\kk)$ also remain unchanged.
Hence, results for generalized vertex spaces follow from the case when
$\mu(\sss)=1$.    
\end{remark}

\subsection{Integral operators}

Given a kernel $\kk$ on a measure space $(\sss,\mu)$,
let $\tk$ be the integral operator on $(\sss,\mu)$ with kernel $\kk$,
defined by
\begin{equation}
  \label{tk}
(\tk f)(x)\=\int_\sss \kk(x,y)f(y)\dd\mu(y),
\end{equation}
for any (measurable) function $f$ such that this integral is defined (finite or
$+\infty$) for \aex{} $x$.
(As usual, we shall assume without comment that all functions
considered are measurable.) 
Note that $\tk f$ is defined for every $f\ge0$, with $0\le\tk f\le\infty$. 

We define
\begin{equation}\label{tnorm}
  \norm{\tk}\=\sup\bigset{\normll{\tk f}: f\ge0,\,\normll{f}\le1}
\le\infty.
\end{equation}
When finite, $\norm{\tk}$ is the norm of $\tk$ as an operator in
$L^2(\sss,\mu)$.
We denote the inner product in
(real) $L^2(\mu)$ by 
$\innprod{f,g}=\innprodmu{f,g}\=\ints fg\dd\mu$,
and the norm by $\normll{f}\=\innprodmu{f,f}\qq$.

One of the results of \cite{kernels} is that the function 
$\rho_\kk(x)=\rho(\kk;x)$ is the unique maximal solution to the
non-linear functional equation
\begin{equation}  \label{phik}
 f = 1-e^{-\tk f},
\qquad f\ge0.
\end{equation}
Moreover, if $\norm\tk\le1$, 
then $\rho_\kk=0$ and thus $\rhoka=0$,
while  if $\norm\tk>1$, 
then $\rho_\kk>0$ on a set of positive measure and thus $\rhoka>0$.
(This extends to generalized vertex spaces by the
renormalization in \refR{Renorm}; note that $\rho_\kk$, $\tk$ and $\norm{\tk}$ 
are not changed by the renormalization.) 
The three cases $\norm\tk<1$, $\norm\tk=1$ and $\norm\tk>1$,
are called \emph{subcritical}, \emph{critical} and
\emph{supercritical}, respectively. 

Given a kernel $\kk$ on a type space $(\sss,\mu)$,
let $\muq$ be the measure on $\sss$ defined by
\begin{equation}
  \label{muq}
\dd\muq(x)\=(1-\rho(\kk;x))\dd\mu(x).
\end{equation}
(This is interesting mainly when $\kk$ is supercritical, since otherwise
$\muq=\mu$.)
The \emph{dual kernel} $\kkq$ is the kernel on $(\sss,\muq)$ that
is equal to $\kk$ as a function.\
We %usually 
regard $\tkq$ as an operator acting on the corresponding 
space $L^2(\muq)$. 
Then $\norm\tkq\le1$; typically $\norm\tkq<1$
when $\kk$ is supercritical,
but equality is possible,
see \cite[Theorem 6.7 and Example 12.4]{kernels}.

Note the explicit formula
\begin{equation}
  \label{tkq}
(\tkq f)(x)
\=\int_\sss \kkq(x,y)f(y)\dd\muq(y)
=\int_\sss \kk(x,y)f(y)(1-\rho(\kk;y))\dd\mu(y),
\end{equation}
\ie, $\tkq f = \tk (f(1-\rhokk))$.
Note also that
\begin{equation}\label{muqs}
  \muq(\sss)=\ints(1-\rho(\kk;x))\dd\mu(x) = \musss-\rho(\kk);
\end{equation}
if $\musss=1$, this is the extinction probability of $\bpk$.

\subsection{Small components}

Let $\nk(G)$ denote the number of vertices in components of order
$k$ in a graph $G$. (Thus the number of such components is
$\nk(G)/k$.)
We can write the definition \eqref{sus} as
\begin{equation}\label{susnk}
  \sus(G)=\frac1{|G|}\sumk \frac{\nk(G)}k k^2=\sumk k\frac{\nk(G)}{|G|}.
\end{equation}

By \cite[Theorem 9.1]{kernels}, 
if $(\kkn)$ is a graphical sequence of kernels on a vertex space 
$\vxs$ with limit $\kk$ and $G_n=\gnxx{\kkn}$, then, for every fixed
$k\ge1$, with $\ngek:=\sum_{j\ge k}\nj$ and 
$\rhogek:=\sum_{k\le j\le\infty}\rhoj$, we have
\begin{equation}
  \label{ngeklim}
\ngek(G_n)/n\pto\rhogek(\kk),
\end{equation}
and thus
\begin{equation}
  \label{nklim}
\nk(G_n)/n\pto\rhok(\kk).
\end{equation}
This extends to generalized vertex spaces by normalization 
(if necessary first conditioning on $\xss$) as discussed in
\cite[Subsection 8.1]{kernels}.
Furthermore, \eqref{nklim} holds also on an \iid{} vertex space 
for a constant sequence $\kkn=\kk$, with $\kk$
integrable, by \cite[Lemma 21]{clustering}.

Even more generally, by~\cite[Lemma 2.11]{cutsub}, the same 
conclusions hold when $G_n=G(A_n)$ with $\dcut(A_n,\ka)\to 0$,
and hence when $G_n=G(A_n)$ with $\dcut(A_n,\ka)\pto 0$;
this implies the two special cases above.

\subsection{The giant component}

If $\kk$ is \qir{} (see \cite{kernels} for the definition), 
then under any of our assumptions we have
\begin{equation}
  \label{giant}
|\cc1(G_n)|/n\pto\rho(\kk)
\end{equation}
and 
\begin{equation}
  \label{2nd}
|\cc2(G_n)|/n\pto 0;
\end{equation}
see \cite[Theorems 3.1 and 3.6]{kernels} or \cite[Theorem 1.1]{cutsub}.

\subsection{Monotonicity}
We note a simple monotonicity for $\sus$; there is no corresponding result for $\susq$.
\begin{lemma}
  \label{Lmon}
If $H$ is a subgraph of $G$ with the same vertex set, then $\sus(H)\le\sus(G)$.
\end{lemma}
\begin{proof}
Immediate from the definition \eqref{sus1}.
\end{proof}

\section{Branching processes}\label{Sbp}

For branching processes, as is well-known, the mean cluster size can
be expressed in terms of the operators $\tk$ and $\tkq$.
We write $1$ for the constant function $1$ on $\sss$. 

\begin{lemma}
  \label{LBP1}
For any integrable kernel
$\kk$ on a type space $(\sss,\mu)$ we have
\begin{align}
  \sus(\ka;x)
&=\sumj\tk^j1(x),
\label{bp1a}
\\
 \sus(\ka)
 &=\musssqw\sumj\ints\tk^j1(x)\dd\mu(x)
=\musssqw\sumj\innprodmu{\tk^j1,1},
\label{bp1b}
\\
  \susq(\ka;x)
&=(1-\rho(\kk;x))\sumj\tkq^j1(x),
\label{bp1c}
\\
 \susq(\ka)
 &=\musssqw\sumj\ints\tkq^j1(x)\dd\muq(x)
=\musssqw\sumj\innprod{\tkq^j1,1}_{\muq}.
\label{bp1d}
\end{align}
\end{lemma}

\begin{proof}
  Let $f_j(x)$ be the expected size of generation $j$ in $\bpk(x)$.
Then, for every $j\ge0$, by conditioning on the first generation,
\begin{equation*}
  f_{j+1}(x)=\ints f_j(y)\kk(x,y)\dd\mu(y) = \tk f_j(x),
\end{equation*}
and thus, by induction, $f_j=\tk^j f_0 = \tk^j 1$. Hence, \eqref{bp1a}
follows by summing. Recalling the definition \eqref{suskk}, 
relation \eqref{bp1b} follows immediately.

It is easy to see that if we condition $\bpk(x)$ on extinction, we
obtain another similar branching process $\bpkq(x)$ with $\mu$
replaced by $\muq$. 
Hence, $\tk$ is replaced by $\tkq$, and
\eqref{bp1c} follows from 
\begin{equation*}
  \begin{split}
  \E\bigpar{|\bpk(x)|;{|\bpk(x)|<\infty}}
&=(1-\rho(\kk;x))\E\bigpar{|\bpk(x)|\,\bigl|\,{|\bpk(x)|<\infty}}	
\\&
=(1-\rho(\kk;x))\E\bigpar{|\bpkq(x)|}	
  \end{split}
\end{equation*}
and \eqref{bp1a}. 
Finally, \eqref{bp1d} follows by \eqref{suskkq} and integration,
recalling \eqref{muq}.
\end{proof}

Often, it is convenient to assume for simplicity that $\musss=1$.
\begin{lemma}\label{Lsusq}
Let $\ka$ be an integrable kernel on a type space $(\sss,\mu)$ with $\musss=1$. Then
\begin{equation*}
\susq(\kk)
=  
\sumj\innprod{\tkq^j1,1}_{\muq}
=\muq(\sss)\sus(\kkq)
=(1-\rho(\kk))\sus(\kkq).
\end{equation*}  
\end{lemma}

\begin{proof}
  Use \eqref{bp1d} for $\kk$ and $\mu$
and \eqref{bp1b} for $\kkq$ and $\muq$, together with \eqref{muqs}.
\end{proof}

\begin{theorem}
  \label{TBP1}
Let $\kk$ be an integrable kernel on a type space $(\sss,\mu)$
with $\mu(\sss)=1$.
  \begin{romenumerate}[-10pt]
\item
If $\kk$ is subcritical, i.e., $\norm{\tk}<1$, then
$\sus(\kk;x)=(I-\tk)\qw1$ \aex, and 
$\sus(\kk)=\innprodmu{(I-\tk)\qw1,1}<\infty$.	
\item\label{TBP1b}
Suppose that $\kk$ is supercritical, i.e., $\norm{\tk}>1$,
and also that $\norm{\tkq}<1$.
Then
$\susq(\kk;x)=(1-\rhokk)(I-\tkq)\qw1$ \aex, and 
$\susq(\kk)=\innprodmuq{(I-\tkq)\qw1,1}<\infty$.	
\end{romenumerate}
The conditions of \ref{TBP1b} hold whenever $\norm{\tk}>1$, $\kk$ is \qir, and
$\intss\kk^2<\infty$. 
\end{theorem}

\begin{proof}
An immediate consequence of \refL{LBP1}, since in these cases 
the sums $\sumj\tk^j=(I-\tk)\qw$ and $\sumj\tkq^j=(I-\tkq)\qw$,
respectively, converge as operators on $L^2(\mu)$ and $L^2(\muq)$.
For the final statement we use \cite[Theorem 6.7]{kernels}, which
yields $\norm{\tkq}<1$.
\end{proof}

In fact, for the last part one can replace the assumption
that $\intss \ka^2<\infty$ by the weaker assumption that $\tk$
is compact; this
is all that is used in the proof of \cite[Theorem 6.7]{kernels}.

In the critical case, when $\norm\tk=1$, we have
$\sus(\kk)=\susq(\kk)$.
We typically expect the common value to be infinite,
but there are exceptions;
see \refSS{SSCHKNS}.

\begin{theorem}
  \label{Tcritical}
  \begin{thmenumerate}
\item
If $\kk$ is critical and $\tk$ is a compact operator on $\llmu$, then
$\sus(\kk)=\infty$. 
In particular, this applies if
$\int_{\sss^2}\kk(x,y)^2\dd\mu(x)\dd\mu(y)<\infty$. 
\item
 If $\kk$ is supercritical, then $\sus(\kk)=\infty$.
  \end{thmenumerate}
\end{theorem}

\begin{proof}
\pfitem{i}
If $\int_{\sss^2}\kk^2<\infty$, then $\tk$ is a \HS{}
operator and thus compact. 

$\tk$ is always self-adjoint (when it is bounded), 
so if $\tk$ is compact and critical, then it has 
an eigenfunction $\psi$ with eigenvalue $\norm\tk=1$; moreover,
the eigenspace has finite dimension and there is at least one such
eigenfunction $\psi_1\ge0$ (with $\normll{\psi_1}=1$, say), 
see Lemma 5.15 in \cite{kernels} and its proof, where only compactness is used.
There may also be eigenfunctions with eigenvalue $-1$, so we consider
the positive compact operator
$\tk^2$ and let $\psi_1,\dots,\psi_m$ be an orthonormal basis of the eigenspace
for the eigenvalue 1 of $\tk^2$. The orthogonal complement is also
invariant, and $\tk^2$ acts there with norm $R<1$. Hence,
\begin{equation*}
  \innprod{\tk^{2n}1,1}
=\sum_{i=1}^m \innprod{1,\psi_i}^2 + O(R^n)
\to \sum_{i=1}^m \innprod{1,\psi_i}^2.
\end{equation*}
Since the terms in the sum are non-negative and
$\innprod{1,\psi_1}=\int\psi_1\dd\mu>0$, the limit is strictly
positive and thus 
$\sumj\innprod{\tk^j1,1}$ cannot converge. Since the terms in this sum
are non-negative, \eqref{bp1b} yields
$\sus(\kk)=\musssqw\sumj\innprod{\tk^j1,1}=\infty$.

\pfitem{ii} 
By \cite[Theorem 6.1]{kernels} we have $\P(|\bpk|=\infty)=\rho(\kk)>0$,
so $\sus(\kk)=\infty$.
\end{proof}

In the subcritical case, we can
find $\sus(\kk)$ by finding $(I-\tk)\qw1$, \ie{}, by solving the
integral equation $f=\tk f+1$. Actually, we can do this for any $\kk$,
and can use this as a test of whether $\sus(\kk)<\infty$.

\begin{theorem}\label{Teq}
Let $\kk$ be a kernel on a type space $\sssmu$. Then 
the following are equivalent:
  \begin{romenumerate}
\item
$\sus(\kk)<\infty$.
\item
There exists a function $f\ge0$ in $L^1(\mu)$ such that (\aex)
\begin{equation}
  \label{em}
f = Tf + 1.
\end{equation}
\item
There exists a function $f\ge0$ in $L^1(\mu)$  such that (\aex)
\begin{equation}
  \label{emw}
f \ge Tf + 1.
\end{equation}
  \end{romenumerate}
When the above conditions hold, there is a smallest non-negative solution $f$ to \eqref{em},
that is also a smallest non-negative solution to \eqref{emw}; this
minimal solution $f$ equals $\sus(\kk;x)$, and thus 
$\sus(\kk)=\musssqw\ints f\dd\mu$.
\end{theorem}

\begin{proof}
Recalling \eqref{bp1a}, let $g(x)\=\sus(\kk;x)=\sumj\tk^j 1(x)$; 
this is a function $\sss\to[0,\infty]$ with
$\tk g=\sumji\tk^j1=g-1$, so $g$ satisfies both  \eqref{em} and
\eqref{emw}. %, although it may be infinite.
Further, $\ints g\dd\mu=\musss \sus(\kk)$ by \eqref{bp1b}.
Hence, if (i) holds, %$\sus(\kk)<\infty$, 
then $g\in L^1(\mu)$; 
consequently, $g$ satisfies (ii) and (iii).
(Note that then $g$ is finite a.e.)

Conversely, if $f\ge0$ solves \eqref{em} or \eqref{emw}, then, by
induction,
\begin{equation*}
  f \ge \sum_{j=0}^{n-1} \tk^j 1 + \tk^n f
\end{equation*}
for every $n\ge1$.
Thus $f \ge \sum_{j=0}^{n-1} \tk^j 1$, and 
letting \ntoo{} yields $f\ge g$. Hence, if (ii) or
(iii) holds, then $g\in L^1(\mu)$, and (i) holds. Further, in this
case, $f\ge g$, which shows that $g$ is the smallest solution in both
(ii) and (iii), completing the proof.  
\end{proof}

Note that in the subcritical case, 
\eqref{em} always has a solution in $L^2(\mu)$; \cf{} \refT{TBP1}.
In \refSS{SSCHKNS}, we give an example where $\kk$ is critical and
\eqref{em} has a solution that belongs to $L^1(\mu)$, but
not to $L^2(\mu)$. 
(We do not know whether there can be a non-negative solution in 
$L^2(\mu)$ with $\kk$ critical.)
Moreover, in this example, both in subcritical and
critical cases, there is more than one non-negative solution in $L^1(\mu)$.
However, we can show that there is never more than one non-negative solution in $L^2(\mu)$.
\begin{corollary}\label{Ceq}
Suppose that there exists a function $f\ge0$ in $L^2(\mu)$ such that 
\eqref{em} holds. Then $f$ is the unique non-negative solution  to
\eqref{em} in $L^2(\mu)$,
$\sus(\kk;x)=f(x)$ and 
$\sus(\kk)=\musssqw\ints f\dd\mu$.  
\end{corollary}

\begin{proof}
Let $g$ be the smallest non-negative solution, guaranteed to exist
by \refT{Teq}, and let $h=f-g\ge0$. Since $0\le h\le f$, $h\in L^2(\mu)$.
Then $Th=Tf-Tg=(f-1)-(g-1)=h$, and
  \begin{equation*}
	\begin{split}
\innprod{f,h}
=\innprod{Tf+1,h}	  
=\innprod{Tf,h}+\innprod{1,h}	  
=\innprod{f,Th}+\innprod{1,h}	  
=\innprod{f,h}+\innprod{1,h}.
	\end{split}
  \end{equation*}
Hence $0=\innprod{1,h}=\int h\dd\mu$, so $h=0$ \aex, and $f=g$.	  
\end{proof}

\section{Main results}\label{Smain}

We begin with a general asymptotic lower bound for the susceptibility. 
This bound depends only on convergence of the number of vertices in components
of each fixed size, so it applies under any of the assumptions described
above. More precisely, we state the results in the setting of \refSSS{sss_cc}; as noted
there they then apply (by conditioning) to $\gnkxn$ under the assumptions
in \refSSS{sss_gnkx} or \refSSS{sss_iid}.
As usual, we say that $G_n$ has a certain property 
{\em with high probability}, or {\em \whp{}}, if the probability 
that $G_n$ has this property tends to 1 as $n\to\infty$.

Recall that a matrix denoted $A_n$ is assumed to be symmetric, $n$-by-$n$ and
to have non-negative entries.
\begin{theorem}
  \label{Tlower}
Let $\ka$ be a kernel and $(A_n)$ a sequence of (random)
matrices with $\dcut(A_n,\ka)\pto 0$, 
and set $G_n=G(A_n)$. 
Alternatively, let $G_n=\gnkxn$ satisfy the assumptions
of \refSSS{sss_gnkx} or \refSSS{sss_iid}. Then,
\begin{romenumerate}
  \item
for every $b<\sus(\kk)$, \whp{} $\sus(G_n)>b$, and
  \item
for every $b<\susq(\kk)$, \whp{} $\susq(G_n)>b$.
\end{romenumerate}
Moreover, $\liminf\E\sus(G_n)\ge\sus(\kk)$ and
$\liminf\E\susq(G_n)\ge\susq(\kk)$.
\end{theorem}

\begin{proof}
As noted in \refSSS{sss_cc}, after reducing to the vertex space case
if necessary (and so assuming without loss of generality that
$\musss=1$) it suffices to consider the case $G_n=G(A_n)$.
\pfitem{i}
 Let $K$ be a fixed positive integer. 
Then, by \eqref{susnk},  \eqref{ngeklim} and \eqref{nklim},
\begin{equation*}
  \begin{split}
\sus(G_n)	
&\ge  \sumk (k\bmin K) \frac{\nk(G_n)}{n}
\\&
=
\sum_{k=1}^{K-1} k \frac{\nk(G_n)}{n}
+K \frac{\ngeK(G_n)}{n}
\\&
\pto 
\sum_{k=1}^{K-1} k \rhokkk+K\rhogeK(\kk)
=\sum_{1\le k\le\infty}(k\bmin K) \rhokkk.
  \end{split}
\end{equation*}
As $K\to\infty$, the \rhs{} tends to $\sus(\kk)$ by monotone
convergence and \eqref{suskk}; hence we can choose a finite $K$ such
that the \rhs{} is greater than $b$, and (i) follows.

\pfitem{ii}
By \eqref{sus} and \eqref{susq}, if $\cc1$ is the largest component of
$G_n$ and $\ccc1>K$, then
\begin{equation*}
  \begin{split}
\susq(G_n)	
\ge  \sum_{k=1}^K k \frac{\nk(G_n)}{n}.
  \end{split}
\end{equation*}
On the other hand, if $\ccc1\le K$, then
\begin{equation*}
  \begin{split}
\susq(G_n)
=\sus(G_n)-\ccc1^2/n
\ge\sus(G_n)-K^2/n.
  \end{split}
\end{equation*}
Hence, in both cases,
using \eqref{nklim} again,
\begin{equation}%\label{ta2}
  \begin{split}
\susq(G_n)
\ge  \sum_{k=1}^K k \frac{\nk(G_n)}{n}-\frac{K^2}n
\pto \sum_{k=1}^K k \rhokkk.
  \end{split}
\end{equation}
As $K\to\infty$, the \rhs{} tends to $\susq(\kk)$, and thus we can
choose $K$ such that it exceeds $b$, and (ii) follows.

\pfitem{iii} An immediate consequence of (i) and (ii).
\end{proof}

We continue with a simple general probability exercise.

\begin{lemma}
  \label{LA1}
Let $X_n$ be a sequence of non-negative random variables and suppose
that $a\in[0,\infty]$ is such that
\begin{romenumerate}
  \item
for every real $b<a$, \whp{} $X_n\ge b$, and
\item
$\limsup \E X_n \le a$.
\end{romenumerate}
Then $X_n\pto a$ and $\E X_n\to a$.
Furthermore, if $a<\infty$, then $X_n\lto a$, \ie, $\E|X_n-a|\to0$.
\end{lemma}

\begin{proof}
  If $a=\infty$, (i) says that $X_n\pto\infty$; this
implies $\liminf\E X_n\ge b$ for every $b<\infty$, and
thus
$\E X_n\to\infty$.

Assume now that $a<\infty$, and let $\eps\ge0$. Then, for every $b<a$, by
(i),
\begin{equation*}
  \begin{split}
\E(X_n-a) 
&\ge \eps \P (X_n\ge a+\eps) - (a-b)\P(a+\eps>X_n\ge b) - a\P(X_n<b)
\\	
&\ge \eps \P (X_n\ge a+\eps) - (a-b) - o(1).
  \end{split}
\end{equation*}
Hence
\begin{equation*}
\limsup\E(X_n-a) \ge \eps \limsup\P (X_n\ge a+\eps) - (a-b)
\end{equation*}
and thus, since $b<a$ is arbitrary,
\begin{equation*}
\limsup\E(X_n-a) \ge \eps \limsup \P (X_n\ge a+\eps).
\end{equation*}
Since $\limsup \E (X_n-a)\le 0$ by (ii), this yields 
$\limsup \P (X_n\ge a+\eps)=0$ for every $\eps>0$, which together with
(i) yields $X_n\pto a$.

Moreover, the same argument yields, for every $\eps\ge0$,
\begin{equation*}
\liminf\E(X_n-a) \ge \eps \liminf \P (X_n\ge a+\eps).
\end{equation*}
Taking $\eps=0$ we obtain $\liminf \E X_n \ge a$, which together with
(ii) yields $\E X_n\to a$.
\end{proof}

The idea is to use \refL{LA1} with $X_n=\sus(G_n)$ and $a=\sus(\kk)$ 
or $X_n=\susq(G_n)$ and $a=\susq(\kk)$; then condition (i) is
satisfied by \refT{Tlower}, and we only have to verify the upper bound
(ii) for the expected susceptibility.
For convenience, we state this explicitly.

\begin{lemma}
  \label{Ltest}
Let $\ka$ and $G_n$ be as in \refT{Tlower}.
\begin{romenumerate}
  \item
If\/ $\limsup\E\sus(G_n)\le\sus(\kk)$, then
$\sus(G_n)\pto\sus(\kk)$  and  $\E\sus(G_n)\to\sus(\kk)$.
  \item
If\/ $\limsup\E\susq(G_n)\le\susq(\kk)$, 
then $\susq(G_n)\pto\susq(\kk)$  and 
$\E\susq(G_n)\to\susq(\kk)$.
\end{romenumerate}
\end{lemma}
\begin{proof}
  By \refT{Tlower} and \refL{LA1} as discussed above.
\end{proof}

Sometimes we can control the expectation only after conditioning on some 
(very likely) event. This still gives convergence in probablity. 

\begin{lemma}
  \label{Ltestcond}
Let $\ka$ and $G_n$ be as in \refT{Tlower}, and let $\cE_n$ be an event (depending on $G_n$)
such that $\cE_n$ holds whp.
\begin{romenumerate}
  \item
If\/ $\limsup\E(\sus(G_n);\cE_n) \le\sus(\kk)$, then
$\sus(G_n)\pto\sus(\kk)$.
  \item
If\/ $\limsup\E(\susq(G_n);\cE_n)\le\susq(\kk)$, 
then $\susq(G_n)\pto\susq(\kk)$.
\end{romenumerate}
\end{lemma}
\begin{proof}
After conditioning on $\cE_n$, we still have $\nk(G_n)/n\pto\rhokkk$ for each fixed $k$,
which is all that was needed in the proof of \refT{Tlower}. Letting $\phi=\sus$ or
$\susq$, since $\E(\phi(G_n)\mid \cE_n)\sim \E(\phi(G_n);\cE_n)$,  under the relevant assumption
\refL{LA1} tells us that the distribution of $\phi(G_n)$ conditioned on $\cE_n$
converges in probability to $\phi(\kk)$. But then the unconditional distribution converges in probability.
\end{proof}

We begin with a trivial case, which follows immediately from
\refL{Ltest}.

\begin{theorem}
  \label{Tinfty}
Let $\ka$ and $G_n$ be as in \refT{Tlower}.
\begin{romenumerate}
  \item
If $\sus(\kk)=\infty$, then $\sus(G_n)\pto\infty$  and 
$\E\sus(G_n)\to\infty$.
In particular, this holds if $\kk$ is critical and $\tk$ is compact,
or if $\kk$ is supercritical.
  \item
If $\susq(\kk)=\infty$, then $\susq(G_n)\pto\infty$  and 
$\E\susq(G_n)\to\infty$.
\end{romenumerate}
\end{theorem}

\begin{proof}
  The  extra conditions in \refL{Ltest} are vacuous. For (i), we use
  also \refT{Tcritical}.
\end{proof}

One way to obtain the required upper bound on the susceptibility is by
counting paths. Let $P_\ell=P_\ell(G)$ denote the number of paths
$v_0v_1\dots v_\ell$ of length $\ell$ in the graph $G$. 

\begin{lemma}
\label{Lpaths}
  Let $G$ be a graph with $n$ vertices.
Then
$\sus(G)\le\suml P_\ell(G)/n$.
\end{lemma}

\begin{proof}
  For each ordered pair $(v,v')$ of vertices of $G$ with $v$ and $v'$
  in the same component, there is at least one path (of length $\ge0$)
starting at $v$ and ending at $v'$. 
Thus, counting all such pairs, $\sum_i\ccci^2\le\suml P_\ell$.
\end{proof}

So far our arguments relied only on convergence of the number
of vertices in components of a fixed size $k$, and so
apply in very great generality.
Unfortunately, bounding $\sus(G)$ from above, via \refL{Lpaths}
or otherwise, involves proving bounds for all $k$ simultaneously.
These bounds do not hold in general; we study two special cases
where they do in the next two subsections.

\subsection{Bounded kernels on general vertex spaces}

In this section we consider $G_n=\gnkxn$, where 
$(\ka_n)$ is any uniformly bounded
graphical sequence of kernels on a (generalized)
vertex space $\cV$ with limit $\kk$.
In fact, we shall be consider the more general situation where
$G_n=G(A_n)$ for some sequence $(A_n)$ of uniformly bounded
(random) matrices with $\dcut(A_n,\ka)\pto 0$. 
From the remarks in~\cite{cutsub}, the graphs $\gnkxn$ are of this form.
Note that this is the setting in which the component
sizes were studied by Bollob\'as, Borgs, Chayes and Riordan~\cite{QRperc}.

\begin{theorem}
  \label{Tbounded}
Let $\ka$ be a kernel and $(A_n)$ a sequence of uniformly bounded
matrices with $\dcut(A_n,\ka)\pto 0$, and set $G_n=G(A_n)$.
Alternatively, let $G_n=\gnkxn$ satisfy the assumptions
of \refSSS{sss_gnkx} or \refSSS{sss_iid},
with the $\ka_n$ uniformly bounded.
\begin{romenumerate}
\item
We have $\sus(G_n)\pto\sus(\kk)$. 
\item
If $\kk$ is \qir, then $\susq(G_n)\pto\susq(\kk)$.
\end{romenumerate}
\end{theorem}

The boundedness assumption is essential unless further conditions are imposed;
see \refE{Ebad}.
The extra assumption in (ii) 
is needed to rule out the possibility that there are
two or more giant components, living in different parts of the type space.

\begin{proof}
As noted above, the case of a generalized vertex space $\cV$
may be reduced to the case of a vertex space by conditioning and
renormalization, see \refSSS{sss_gnkx} and
\refR{Renorm}, and the vertex space case 
in  \refSSS{sss_gnkx} or \refSSS{sss_iid}
is a special case of
the version with matrices $A_n$, so it suffices to consider
the latter version. In particular, we may assume that $\mu(\sss)=1$.

Coupling appropriately, we may and shall assume that $\dcut(A_n,\ka)\to 0$.
It is easily seen that this and the uniform boundedness of the $A_n$
imply that $\kk$ is bounded.

For (i), suppose first that $\norm{\tk}\ge 1$. Then, since $\tk$ is compact,
by \refT{Tcritical} we have $\sus(\ka)=\infty$, and by \refT{Tinfty}
we have $\sus(G_n)\pto\infty$ as required.

Suppose then that $\norm{\tk}<1$.
Let $\ka_n=\ka_{A_n}$ denote the piecewise constant kernel corresponding to $A_n$.
Then, letting $1$ denote the vector $(1,\dots,1)$, and writing
$A_n=(a_{ij}^{(n)})$, we have
\begin{equation}\label{s2}
  \begin{split}
\E P_\ell(G_n)
&\le	
\E \sum_{j_0,\dots,j_\ell =1}^n \prod_{i=1}^\ell \frac{a_{j_{i-1},j_i}^{(n)}}n
\\
&=
n \E \int_{\sss^{\ell+1}} \prod_{i=1}^\ell \kk_n(x_{i-1},x_i) \dd\mu(x_0)\cdots\dd\mu(x_\ell)
\\
&=n\innprodmu{T_{\ka_n}^\ell1,1}.
  \end{split}
\end{equation}

Recall that $\ka_n$ and $\ka$ are uniformly bounded, and $\dcut(\ka_n,\ka)\to 0$. 
As noted in~\cite{QRperc}, 
or by the Riesz--Thorin interpolation theorem
\cite[Theorem VI.10.11]{Dunford-Schwartz} 
(for operators $L^\infty\to L^1$ and $L^1\to L^\infty$),
it is easy to check that this implies
$\norm{\tkn}\to\norm{\tk}$. (In fact, the normalized spectra converge;
see~\cite{BCLSV3}.)
Since $\norm{\tk}<1$, it follows that for some $\delta>0$
we have $\norm{\tkn}<1-\delta$ for $n$ large enough, so
$\sum_\ell \innprodmu{T_{\ka_n}^\ell1,1} \le \sum_\ell \norm{\tkn}^\ell$
converges geometrically.

For a fixed $\ell$, and kernels $\ka$, $\ka'$ bounded by $M$, say,
it is easy to check that
$ | \innprodmu{T_{\ka'}^\ell1,1} - \innprodmu{T_{\ka}^\ell1,1}|
\le \ell M^{\ell-1}\cn{\ka'-\ka}$ (see, for example, \cite[Lemma 2.7]{cutsub}).
Since $\innprodmu{T_{\ka'}^\ell1,1}$ is preserved by rearrangement,
we may replace $\cn{\ka'-\ka}$ by $\dcut(\ka',\ka)$ in this bound.
Hence, for each $\ell$, we have
$\innprodmu{\tkn^\ell1,1} \to \innprodmu{\tk^\ell1,1}$.
Combined with the geometric decay established above, it follows that
\[
 \sum_{\ell=0}^\infty \innprodmu{\tkn^\ell1,1} \to \sum_\ell \innprodmu{\tk^\ell1,1} = \sus(\ka).
\]
By \refL{Lpaths} and \eqref{s2} we thus have
\begin{equation*}
\limsup\E\sus(G_n)
\le\limsup \frac1n\suml\E P_\ell(G_n)
\le\limsup \sum_{\ell=0}^\infty \innprodmu{\tkn^\ell1,1} 
 = \sus(\ka),
\end{equation*}
which with \refL{Ltest}(i) gives $\sus(G_n)\pto \sus(\ka)$ as required.

We now turn to $\susq$, i.e., to the proof of (ii).
If $\norm{\tk}\le 1$, then $\rho(\ka)=0$ and
$\susq(\ka)=\sus(\ka)$.
On the other hand, $\susq(G_n)<\sus(G_n)$, so the bound
above gives $\limsup\E\susq(G_n)\le \sus(\ka)=\susq(\ka)$,
and \refL{Ltest}(ii) gives the result.

Now suppose that $\norm{\tk}>1$.
Let $\tG_n$ be the graph obtained from $G_n$ by deleting all vertices
in the largest component $\cc1$, and let $\tn$ be the number
of vertices of $\tG_n$. By the duality result of \cite{cutdual}
(see also \cite[Theorem 12.1]{kernels} for the case $G_n=\gnkxn$),
there is a random sequence $(B_n)$ of matrices
(of random size $\tn\times\tn$)
with $\dcut(B_n,\kkqq)\pto 0$,
such that $\tG_n$ may be coupled to agree \whp{} with $G(B_n)$;
here $\kkqq\=\kkq'$ is $\kkq$ renormalized as
in \eqref{mu'}. (Recall that $\kkq$ is regarded as a
kernel on $(\sss,\muq)$, where $\muq$ defined by
\eqref{muq} is not a probability measure.) By
\refR{Renorm}, $\sus(\kkqq)=\sus(\kkq)$.
 
Note that
\begin{equation}
  \label{fb}
\frac{|\tG_n|}n=\frac{n-\ccc1}n\pto1-\rho(\kk)
\end{equation}
by \eqref{giant}.
After conditioning
on the number of vertices of $\tG_n$ and the matrices $B_n$,
we can apply part (i) to conclude that
\begin{equation}  \label{fa}
 \sus(\tG_n)= \sus(G(B_n))+\op(1)\pto\sus(\kkqq)=\sus(\kkq).
\end{equation}
Finally, if $\set{\cci}_{i\ge1}$ are the components of $G_n$,
then $\set{\cci}_{i\ge2}$ are the components of $\tG_n$, and thus by
\eqref{susq}, \eqref{sus}, \eqref{fb}, \eqref{fa}
and \refL{Lsusq}
\begin{equation*}
\susq(G_n)
= \frac{\sum_{j\ge2}\ccc{i}^2}{n}	
= \frac{|\tG_n|\sus(\tG_n)}{n}
\pto (1-\rho(\kk)) \sus(\kkq)
=\susq(\kk).
\qedhere
\end{equation*}
\end{proof}

\subsection{The \iid{} case} 

\begin{theorem}
  \label{Tiid}
Let $\kk$ be an
integrable kernel on an \iid{} vertex space $\cV$.
Then $\sus(\gnkx)\pto\sus(\kk)$ and $\E\sus(\gnkx)\to\sus(\kk)$.
\end{theorem}

\begin{proof}
Similarly to the estimate in the proof of \refT{Tbounded},
for any $\ell$, the expected number $\E P_\ell$ of paths of length $\ell$ is
  \begin{multline*}
n\dotsm(n-\ell)
\int_{\sss^{\ell+1}}\prod_{i=1}^\ell \min\Bigpar{\frac{\kk(x_{i-1},x_{i})}n,1}
\dd\mu(x_0)\dotsm\dd\mu(x_{\ell})
\\
\le
n\int_{\sss^{\ell+1}}\prod_{i=1}^\ell \kk(x_{i-1},x_{i})
\dd\mu(x_0)\dotsm\dd\mu(x_{\ell})
=n\innprodmu{\tk^\ell1,1}.
  \end{multline*}
Summing over all $\ell\ge0$, we see by \eqref{bp1b} that the expected
total number of paths is at most $n\sus(\kk)$.
Hence, by \refL{Lpaths},
\begin{equation}\label{tiid}
 \E\sus(\gnkx) \le \E\suml P_\ell/n 
\le \sus(\kk).
\end{equation}
The result follows by \refL{Ltest}.
\end{proof}

Our next aim is to prove a similar result for
$\susq$. Unfortunately, we need an extra assumption. We shall
assume that $\tk$ is compact, though any condition guaranteeing
\eqref{need} below will do.

\begin{theorem}\label{Tsup}
Let $\kk$ be an irreducible, integrable kernel
on an \iid{} vertex space $\cV$ with $\norm\tk>1$, and let $G_n=\gnkx$.
If $\tk$ is compact, then
$\susq(G_n)\pto\susq(\kk)$.
\end{theorem}

We do not know whether compactness, or some similar assumption, is necessary
for this result.

The main idea of the proof is to count the expected number of paths $P$ such that $P$ is not
joined to a large component of $G_n-P$. We start with a few
preparatory lemmas that hold under more general conditions than
\refT{Tsup} itself.

Recall that $\cC_1=\cC_1(G_n)\subseteq [n]$ denotes the (vertex
set of) the largest component of $G_n$. As in~\cite{kernels}, given $G_n$,
let $\nuni$ denote the empirical distribution
of the types of the vertices in $\cC_1(G_n)$, so
for $A\subset \sss$ we have
\begin{equation*}
 \nuni(A) = n^{-1} \bigl|\bigl\{i\in \cC_1(G_n) : x_i\in A\bigr\}\bigr|.
\end{equation*}

\begin{lemma}\label{gcdist}
Let $\kk$ be an irreducible, integrable kernel on an \iid{} vertex space $\vxs=(\sss,\mu,\xss)$,
and let $A$ be a measurable subset of $\sss$. Then
\begin{equation*}
 \nuni(A)\pto \mu_\kk(A)\=\int_A\rho(\kk;x)\dd\mu(x).
\end{equation*}
More precisely, the convergence is uniform in $A$: given any $\eps>0$
there is an $n_0$ such that for all $n\ge n_0$ and all measurable $A$ we have
\begin{equation*}
 \P\bigpar{ | \nuni(A) - \mu_\kk(A) | \ge \eps } \le \eps.
\end{equation*}
\end{lemma}
Note that the first statement corresponds to Theorem 9.10 of~\cite{kernels}, but, 
due to the different conditions, is not implied by it.
\begin{proof}
It suffices to prove the second statement.
Fix $\eps>0$ once and for all,
and choose $k_0$ so that $\rho_{\ge k_0}(\kk) \le \rho(\kk)+\eps/6$;
this is possible since $\rho_{\ge k}(\kk)\downto \rho(\kk)$ as $k\to\infty$.

We start by considering components of a fixed size.
Let $N_k(A)$ denote the number of vertices $i$ of $G_n$ such that $i$ is
in a component of order $k$ and $x_i\in A$.
If $\kk$ is bounded, then using the local coupling argument in
\cite[Section 3]{clustering}
it is easy to check that for each $k$ we
have $N_k(A)/n \pto \rho_k(A)\=\int_A \rho_k(x) \dd\mu(x)$,
uniformly in $A$.
Using the fact that adding or deleting an edge from a graph
$G$ changes the set of vertices in components of size $k$
in at most $2k$ places, and arguing as in~\cite{kernels}, the same statement
for general $\ka$ follows easily.

Summing over $k\le k_0$, we thus have $N_{\le k_0}(A)/n \pto \rho_{\le k_0}(A)$.
In particular,
\begin{equation}\label{sA}
 \P\bigpar{ |N_{\le k_0}(A)/n - \rho_{\le k_0}(A) | \ge \eps/5 } \le \eps/3
\end{equation}
for all large enough $n$ and all measurable $A$.

By a {\em medium component} of $G_n$ we mean any component
of size greater than $k_0$ other than $\cC_1(G_n)$. Let $M$ denote the number of vertices
in medium components, and $M(A)$ the number with types in $A$.
Since $N_k(G_n)/n\pto \rho_k$ for each $k$ and $|\cC_1(G_n)|/n\pto \rho(\kk)$,
we have $M(G_n)/n\pto \rho_{\ge k_0+1}(\kk) -\rho(\kk)\le \eps/6$.
Hence, \whp{}
\begin{equation}\label{mA}
 \sup_A M(A) = M(G_n) \le \eps n/5.
\end{equation}
Let $\#(A)$ denote the number of vertices with types in $A$.
Then $\#(A)$ has a binomial distribution with parameters $n$ and $\mu(A)$,
so for $n$ large enough we have
\begin{equation}\label{aA}
 \P\bigpar{ |\#(A)/n - \mu(A) | \ge \eps/5 } \le \eps/3
\end{equation}
for all $A$.
Finally, let $C_1(A)=n\nuni(A)$ denote the number of vertices in $\cC_1(G_n)$ with types in $A$.
Then
\begin{equation}\label{C1sum}
 C_1(A) = \#(A) - N_{\le k_0}(A) -M(A) +O(1),
\end{equation}
with the final $O(1)$ correction term accounting for the possibility
that $|\cC_1(G_n)|\le k_0$, so the `giant' component is `small'.

Combining equations \eqref{sA}--\eqref{C1sum}, we see that
\begin{equation*}
 \P\bigpar{ |C_1(A)/n - (\mu(A)-\rho_{\le k_0}(A)) | \ge 4\eps/5 } \le \eps
\end{equation*}
for all large enough $n$ and all $A$. But
\begin{equation*}
 \mu(A)- \rho_{\le k_0}(A) = \mu_\kk(A) + \sum_{k=k_0+1}^\infty \rho_k(A).
\end{equation*}
The sum above is at least $0$ but, by choice of $k_0$, at most $\eps/6$, so 
$\mu(A)- \rho_{\le k_0}(A)$ is within $\eps/6$ of $\mu_\kk(A)$ and the result follows.
\end{proof}

In~\cite[Theorem 1.4]{cutsub},
it was shown (in a slightly different setting) 
that stability of the giant component
under deletion of vertices implies that the distribution of the size of the giant
component has an exponential tail. 
Parts of this argument adapt easily to the present setting. 

First, Lemma 1.7 of~\cite{cutsub}
shows that if $\kk$ is a kernel, then the $n$-by-$n$ matrices
obtained by sampling $\kk$ at \iid\ points $x_1,\ldots,x_n$ converge
in probability to $\kk$, with respect to the cut norm. This implies
that all results of~\cite{cutsub} asserting that a certain
conclusion holds \whp{} apply to the corresponding
random graphs (see~\cite[Remark 1.5]{cutsub}).
In particular, Theorem 1.3 of~\cite{cutsub} implies the following result.

\begin{theorem}\label{stab}
Let $\kk$ be an irreducible, integrable kernel
on an \iid{} vertex space $\cV$, and let $G_n=\gnkx$.
For every $\eps>0$ there is a $\delta>0$ such that \whp{} we have
\begin{equation*}
 \rho(\kk)-\eps \le |\cC_1(G_n')|/n \le \rho(\kk)+\eps
\end{equation*}
for every graph $G_n'$ that may be obtained from $G_n$ by deleting
at most $\delta n$ vertices and their incident edges, and then adding
or deleting at most $\delta n$ edges.\qed
\end{theorem}

Using this result, it is easy to get our exponential lower tail bound.
Unfortunately, there is a minor complication, due to the possible (but
very unlikely) non-uniqueness of the giant component. 

Let $\tC_1(A)=\tC_1(A;G_n)$ denote the maximum over components $\cC$ of $G_n$
of the number of vertices of $\cC$ with types in $A$,
so $\tC_1(A)$ is within $|\cC_2(G_n)|$ of $C_1(A)=n\nuni(A)$.

\begin{lemma}\label{gceb} 
Let $\kk$ be an irreducible, integrable kernel on an \iid{} vertex space $\vxs=(\sss,\mu,\xss)$
with $\norm\tk>1$,
and let $\eps>0$. Then there is a $c=c(\kk,\eps)>0$
such that for all large enough $n$, for
every subset $A$ of $\sss$ we have
\begin{equation}\label{nnts}
 \P\bigpar{ \tC_1(A;G_n) \le (\mu_\kk(A)-\eps)n } \le e^{-cn}.
\end{equation}
\end{lemma}
\begin{proof}
Fix $A$. Given a graph $G$ on $[n]$ where each vertex has a type in $\sss$,
let $D(G)=D_A(G)$ be the minimum number of vertices that must be deleted from $G$
so that in the resulting graph $G'$ we have
\begin{equation}\label{tC1}
 \tC_1(A;G')\le (\mu_\kk(A)-\eps)n,
\end{equation}
so our aim is to bound $\P(D(G_n)=0)$.
By Lemma~\ref{gcdist}, \whp{} $\cC_1(G_n)$ has at least $(\mu_\kk(A)-\eps/2)n$
vertices with types in $A$. Also, by \refT{stab}, there is some $\delta>0$
such that \whp{} deleting at most $\delta n$ vertices of $G_n$ removes
less than $\eps n/2$ vertices from the (\whp{} unique) giant component.
It follows that $\E D(G_n) \ge \delta n/2$ for $n$ large;
moreover, this bound is uniform in $A$.

Since the condition \eqref{tC1} is preserved by deleting vertices,
if $G''$ is obtained from $G$ by adding
and deleting edges all of which are incident with one vertex $i$,
and also perhaps changing the type of $i$, then $|D(G)-D(G'')|\le 1$.
We may construct $G_n$ by taking independent variables $x_1,\ldots,x_n$
and $\{y_{ij}: 1\le i<j\le n\}$ all of which are uniform on $[0,1]$,
and joining $i$ to $j$ if and only if $y_{ij}\le \kk(x_i,x_j)/n$.
Modifying the variables in $S_j= \{x_j\}\cup \{y_{ij}: i<j\}$ affects
only edges incident with vertex $j$. Considering the values of all
variables in $S_j$ as a single random variable $X_j$, 
we see that $D(G_n)$ is a Lipschitz function of $n$ independent
variables, so by McDiarmid's inequality~\cite{McD}
we have
\begin{equation*}
 \P\bigpar{ D(G_n)=0 } \le e^{-2 (\E D(G_n))^2/n}  \le e^{-\delta^2 n/2},
\end{equation*}
completing the proof.
\end{proof}

It would be nice to have an exponential bound on the upper tail of the number of 
vertices in `large' components.  Unfortunately, the argument
in~\cite{cutsub} does not seem to go through. Indeed, the
corresponding result is false in this setting without an additional
assumption: it is easy to find a $\ka$ for which there is a small, but
only polynomially small, chance that some vertex $v$ has degree of
order $n$. In this way one can even arrange that $\P(|\cC_1(G_n)|=n)$
is only polynomially small in $n$.

The next lemma is the combinatorial heart of the proof of \refT{Tsup}.
Unfortunately, we cannot bound the expectation of $\susq$ directly, only
the contribution from components up to size some small constant times $n$.
Formally, given a graph $G$ with $n$ vertices and a $\delta>0$, let
\begin{equation}\label{susqd}
  \susqd(G)\=\frac1n\sum_{v\in V(G) \,:\, |\cC(v)|\le\delta n} |\cC(v)|
 = \frac1n\sum_{i\,:\,\ccci\le \delta n} \ccci^2.
\end{equation}
Note that if $\ccc2\le\delta n< \ccc1$, then $\susqd(G)=\susq(G)$.

Given a kernel $\kk$ and an $M>0$, we write $\kk^M$ for the pointwise minimum of $\kk$ and $M$.

\begin{lemma}\label{sc}
Let $\kk$ be an irreducible, integrable kernel on an \iid{} vertex space $\cV$
with $\norm\tk>1$, and let $\eps>0$ and $M>0$.
Then there is a $\delta=\delta(\eps,M,\kk)>0$
such that
\begin{equation*}
 \E \susqd(\gnkx) \le \musssqw\sumj\innprod{\tkqq^j1,1}_{\muqq} +o(1),
\end{equation*}
where $\muqq$ is the measure on $\sss$ defined by
$\dd\muqq(x) = f(x)\dd\mu$ with %\marginal{can remove the $(1-\eps)$ inside if desired. Changed $\eps M$ to $\eps$}
\begin{equation}\label{fdef}
 f(x)=\bigpar{1-\rho\bigpar{(1-\eps)\ka^M;x}+5\eps} \wedge 1,
\end{equation}
and $\tkqq$ is the integral operator on $(\sss,\muqq)$ with kernel $\ka$.
\end{lemma}

\begin{proof}
As usual, we may and shall assume that $\mu(\sss)=1$.

Note that the statement becomes stronger if
we increase $M$ and/or decrease $\eps$.
Thus we may assume that $(1-\eps)\ka^M$ is supercritical,
and that $\rho((1-\eps)\ka^M)>2\eps$.
We also assume that $M>1$ and $e^{4\eps}<1+5\eps$.

Let $0<\delta<\eps/M$ be a small constant to be chosen later, depending only
on $\kk$, $\eps$ and $M$,
and let $N=n\susqd(G_n)$ 
denote the number of ordered pairs $(v,w)$ of vertices
of $G_n=\gnkx$ such that $v$ and $w$ are in a common component of
size at most $\delta n$.
Also, let $N_j$ denote the number of such pairs joined by a path
of length $j$.
Since $N \le \sum_{j=0}^{\delta n-1} N_j$, it suffices
to show that for $0\le j< \delta n$ we have
\begin{equation}\label{Nrb}
 \E N_j/n \le \innprod{\tkqq^j1,1}_{\muqq} +o(1/n),
\end{equation}
with the error bound uniform in $j$.

We may bound $N_j$ by the number of paths of length $j$ in $G_n$
lying in components with at most $\delta n$ vertices.
Thus $\E N_j$ is at most $n^{j+1}$ times the probability
that $12\cdots (j+1)$ forms such a path.
Let $V'$ consist of the last $(1-\eps/M)n$ 
vertices of $G_n$.
Coupling $G_n$ and $G_n^M=G^{\vxs}(n,\ka^M)$ in the usual way
so that $G_n^M\subseteq G_n$, let $G'$ be the subgraph
of $G_n^M$ induced by $V'$, noting that $G'\subset G_n$.
Let $\cA=\cA_j$ 
be the event that $12\cdots (j+1)$ forms
a path in $G_n$, and let $\cB=\cB_j$ be the event
that some vertex in $[j+1]$ is joined by an edge of $G_n^M$ to 
some component of $G'$ of order at least $\delta n$.
Then
\begin{equation*}
 \E N_j \le n^{j+1} \P(\cA\cap \cB^\comp).
\end{equation*}
Unfortunately, we cannot quite prove the estimate we need
for the right hand side above, so instead we use the less
natural but stronger bound
\begin{equation}\label{njub}
 \E N_j \le \binom{n}{j+1} \E( N_j' 1_{\cB^\comp} ),
\end{equation}
where $N_j'$ is the number of ordered pairs $(v,w)$
of vertices in $V_0=[j+1]$ such that $v$ and $w$ are joined
in $G_n$ by a path of length $j$ lying in $V_0$ (and
thus visiting all vertices of $V_0$).

Roughly speaking, the idea is to show that
with very high probability  $\cC_1(G')$ will contain
almost the `right' number of vertices of each type, so that
given the type $y$ of one of the first $j+1$ vertices,
its probability of sending an edge to $\cC_1(G')$
is almost what it should be, namely $\rho((1-\eps/M)\ka^M;y)$.
Unfortunately we cannot achieve this for all $y$,
but we can achieve it for $\{x_1,\ldots,x_{j+1}\}$,
which is all we need. Also, rather than working with $\cC_1(G')$,
we work with the union of all components of order at least $\delta n$.

Let $n'=(1-\eps/M) n$. Ignoring the irrelevant rounding
to integers, $G'$ has the distribution of $\gxxx{n'}{(1-\eps/M)\ka^M}$,
which dominates that of $\gxxx{n'}{(1-\eps)\ka^M}$. 

Recall that $(1-\eps)\ka^M$ is supercritical and that $\rho((1-\eps)\ka^M)>2\eps$.
Applying \refL{gceb} to $G'=\gxxx{n'}{(1-\eps)\ka^M}$ we find that 
there is some $c>0$ such that
for any measurable $A\subset \sss$ we have 
\begin{equation}\label{npAold}
  \P\bigpar{ \tC_1(A;G') \le (\mu'(A)-2\eps/M)n } 
 \le  \P\bigpar{ \tC_1(A;G') \le (\mu'(A)-\eps/M)n' } 
 \le e^{-cn},
\end{equation}
where $\mu'=\mu_{(1-\eps)\kk^M}$.

Let 
\begin{equation*}
 \delta_0=\min\bigset{\eps/M,1/10}>0,
\end{equation*}
and fix $0<\delta<\delta_0$ chosen small enough that
\begin{equation}\label{edd}
 (e/\delta)^\delta<e^{c/2}.
\end{equation}
Let $L$ denote the union of all components of $G'$ of order 
at least $\delta n$,
and let $L(A)$ be the number of vertices
in $L$ with types in $A$.
If $\mu'(A)\ge3\eps/M$ and $\tC_1(A;G')\ge (\mu'(A)-2\eps/M)n$,
then since the final quantity is at least $\delta n$ we have $L(A)\ge\tC_1(A;G')$.
Using \eqref{npAold}, it follows that
\begin{equation}\label{npAnew}
  \P\bigpar{ L(A) \le (\mu'(A)-3\eps/M)n } \le e^{-cn}
\end{equation}
for any $A$; the condition is vacuous if $\mu'(A)< 3\eps/M$.

Given $y\in\sss$ and $i\ge 0$, let $A_{y,i}=\{x\in \sss: \ka^M(x,y)\ge \eps i\}$.
Let $\cE_y$ be the event that
$L(A_{y,i})/n \ge \mu'(A_{y,i})-3\eps/M$ holds for all $i$ with
$1\le i\le M/\eps$.
Applying \eqref{npAnew} $M/\eps=O(1)$ times, we see that
\begin{equation}\label{PEy}
 \P(\cE_y^\comp) \le (M/\eps)e^{-cn}=O(e^{-cn}).
\end{equation}

If $\cE_y$ holds, then
\begin{equation*}
 \sum_{v\in L} \ka^M(x_v,y) \ge \sum_{i=1}^{M/\eps} L(A_{y,i})\eps
 \ge \sum_{i=1}^{M/\eps} \eps (\mu'(A_{y,i}) - 3\eps/M) n
 \ge n\sum_{i=1}^{M/\eps} \eps \mu'(A_{y,i}) - 3\eps n.
\end{equation*}
Now $A_{y,i}$ is empty for $i>M/\eps$, so we have
\begin{multline*}
 \sum_{i=1}^{M/\eps} \eps \mu'(A_{y,i}) = \sum_{i=1}^\infty \eps \mu'\{x: \ka^M(x,y)\ge \eps i\}
 = \int_\sss \eps \floor{\ka^M(x,y)/\eps} \dd\mu'(x) \\
 \ge \int_\sss \ka^M(x,y)\dd\mu'(x) -\eps
 = \int_\sss \ka^M(x,y)\rho((1-\eps)\ka^M;x) \dd\mu(x) -\eps.
\end{multline*}
Putting these bounds together, writing $\ka'$ for $(1-\eps)\ka^M$, we have
\begin{eqnarray*}
 \sum_{v\in L} \ka^M(y,x_v)/n 
 &\ge& 
 \int_\sss \ka^M(x,y)\rho(\ka';x) \dd\mu(x) - 4\eps \\
 &=& (T_{\ka^M}\rho_{\ka'})(y)-4\eps 
 \ge (T_{\ka'}\rho_{\ka'})(y)-4\eps.
\end{eqnarray*}
Recalling that $\ka'$
is supercritical, from \eqref{phik} we have
$T_{\ka'}\rho_{\ka'} = -\log(1-\rho_{\ka'})$, so 
when $\cE_y$ holds we have
\begin{equation*}
  \sum_{v\in L} \ka^M(y,x_v)/n \ge -\log(1-\rho(\ka';y)) -4\eps,
\end{equation*}
and hence
\[
  \prod_{v\in L} (1-\ka^M(y,x_v)/n) \le (1-\rho(\ka';y))e^{4\eps}
 \le 1-\rho(\ka';y) + 5\eps.
\]
Since $\ka^M$ is bounded by $M$, and the product is always at most $1$,
it follows that if $\cE_y$ holds and $n\ge M$, then
\begin{equation}\label{jtor}
  \prod_{v\in L} \bigpar{1- (\ka^M(y,x_v)/n \wedge 1)} \le f(y).
\end{equation}

Let $\cE=\cE_{x_1}\cap \cdots\cap \cE_{x_{j+1}}$. Note that $G'$ is independent
of $x_1,\ldots,x_{j+1}$. Given these types,
from \eqref{PEy} we have $\P(\cE) = 1-O(je^{-cn}) = 1- O(ne^{-cn})$,
with the implicit constant independent of the types. Hence, we have
$\P(\cE)=1-O(ne^{-cn})$ unconditionally.
Then, for $j\le\gd n$,
\begin{equation}\label{noE}
 \binom{n}{j+1}\E( N_j' 1_{\cE^\comp} ) \le \binom{n}{j+1} (j+1)^2 \P(\cE^\comp)
 \le (e/\delta)^{\delta n} n^2 \P(\cE^\comp) = o(1),
\end{equation}
using \eqref{edd} in the last step.

Estimating $N_j'$ by the number of paths of length $j$
lying in $V_0$,
\begin{equation}\label{yesE}
 \binom{n}{j+1}\E( N_j' 1_{\cB^\comp\cap \cE})
  \le \binom{n}{j+1} (j+1)! \P(\cA\cap \cB^\comp\cap \cE) \le n^{j+1} \P(\cA\cap \cB^\comp\cap \cE).
\end{equation}
To estimate the final probability let us condition on $G'$ and also on the vertex
types $x_1,\ldots,x_{j+1}$, assuming as we may that $\cE$ holds.
Note that we have not yet `looked at' edges within $V_0$, or edges from $V_0$ to $V'$.
The conditional probability of $\cA$ is then exactly
\begin{equation*}
 \prod_{i=1}^{j} (\kk(x_i,x_{i+1})/n \wedge 1) \le n^{-j} \prod_{i=1}^j \kk(x_i,x_{i+1}).
\end{equation*}
For each $i\le j+1$, since $\cE_{x_i}$ holds we have from \eqref{jtor}
that the probability that $i$ sends no edge to $L$ is at most
$f(x_i)$.
These events are (conditionally) independent for different $i$, so
\begin{equation*}
 \P(\cA\cap \cB^\comp\cap \cE \mid x_1,\ldots,x_{j+1})
 \le n^{-j}\prod_{i=1}^j \kk(x_i,x_{i+1}) \prod_{i=1}^{j+1} f(x_i).
\end{equation*}
Integrating out we find that
\begin{eqnarray*}
 n^{j+1} \P(\cA\cap \cB^\comp\cap \cE)
 &\le& n\int_{\sss^{j+1}} \prod_{i=1}^j \kk(x_i,x_{i+1}) \prod_{i=1}^{j+1} f(x_i) 
 \dd\mu(x_1)\cdots\dd\mu(x_{j+1}) \\
  &=& n \innprod{\tkqq^j1,1}_{\muqq}.
\end{eqnarray*}
From \eqref{yesE} it follows that $\binom{n}{j+1}\E( N_j' 1_{\cB^\comp\cap \cE}) \le n\innprod{\tkqq^j1,1}_{\muqq}$.
Combined with \eqref{noE} and \eqref{njub} this establishes \eqref{Nrb}; as
noted earlier, the result follows.
\end{proof}

Taking, say, $M=1/\eps$ and defining $f_\eps(x)$ by \eqref{fdef},
as $\eps\to 0$ we have $(1-\eps)\ka^M\upto \ka$ pointwise,
and hence $\rho((1-\eps)\ka^M;x)\upto \rho(\ka;x)$ pointwise.
Thus $f_\eps(x)\downto 1-\rho(\ka;x)$  pointwise.
If we know that 
$\innprod{\tkqq^j1,1}_{\muqq}<\infty$ for some $\eps>0$, then
by dominated convergence it follows that
$ \innprod{\tkqq^j1,1}_{\muqq} \downto  \innprod{\tkq^j1,1}_{\muq}$.
Furthermore, if we have
\begin{equation}\label{need}
 \sumj \innprod{\tkqq^j1,1}_{\muqq}<\infty
\end{equation}
for some $\eps>0$, then
by dominated convergence, as $\eps\to0$ we have
\begin{equation*}
 \sumj\innprod{\tkqq^j1,1}_{\muqq} \downto  \sumj\innprod{\tkq^j1,1}_{\muq} = \susq(\ka).
\end{equation*}
Unfortunately we need some assumption on $\ka$ to establish \eqref{need}.

\begin{proof}[Proof of \refT{Tsup}]
Suppose for the moment that \eqref{need} holds for some $\eps>0$,
where $\muqq$ is defined using $f_\eps(x)$, which is in turn
given by \eqref{fdef} with $M=1/\eps$, say.

By the comments above, it follows that, given any $\eta>0$,
choosing $\eps$ small enough and $M$ large enough
we have $\sumj\innprod{\tkqq^j1,1}_{\muqq}\le\susq(\kk)+\eta$.
\refL{sc} then gives $\E \susqd(G_n)\le \susq(\kk)+2\eta$
if $n$ is large enough, for some $\delta=\delta(\eta)>0$.
Hence, if $\delta=\delta(n)$ tends to zero,
we have
\begin{equation}\label{qdb}
 \limsup \E \susqd(G_n)\le \susq(\kk).
\end{equation}

Since $\kk$ is supercritical we have $\rho(\kk)>0$,
and by \eqref{giant} we have $|\cc1(G_n)|\ge \rho(\kk)n/2$ whp.
For any fixed $\delta>0$, by \eqref{2nd} we have $|\cc2(G_n)|<\delta n$ whp;
this also holds if $\delta=\delta(n)$ tends to zero sufficiently slowly.
Given a function $\delta(n)$, let $\cE_n$ be the event that
$|\cc2(G_n)|\le n \delta(n)< |\cc1(G_n)|$. Then, provided $\delta(n)$
tends to zero slowly enough, $\cE_n$ holds whp.
When $\cE_n$ holds we have $\susqd(G_n)=\susq(G_n)$,
so $\E (\susq(G_n);\cE_n) \le \E\susqd(G_n)$,
and \eqref{qdb} gives $\limsup \E(\susq(G_n);\cE_n)\le \susq(\kk)$.
By \refL{Ltestcond} this implies that $\susq(G_n)\pto\susq(\kk)$, which
is our goal.
It thus suffices to establish that \eqref{need} holds for some $\eps>0$.

Recall that 
$f_\eps(x)\le 1$ and $f_\eps\downto f_0=1-\rho_\ka$ as $\eps\to 0$.
Recall also that
$\tkqq$ is defined as the integral operator 
\[
g\mapsto
\int\kk(x,y)g(y)\dd\muqq(y)=\int\kk(x,y)f_\eps(y)g(y)\dd\mu(y)
\]
on $L^2(\muqq)$. The map
$g(x)\mapsto g(x)f_\eps(x)\qq$ is an isometry of
$L^2(\muqq)$ onto $L^2(\mu)$, and thus
$\tkqq$ is unitarily equivalent to the integral operator 
$T_\eps$ on $L^2(\mu)$
with kernel $f_\eps(x)^{1/2}\ka(x,y)f_\eps(y)^{1/2}$.
In particular,
$\norm{\tkqq}=\norm{T_\eps}$, and for
the special case $\eps=0$, when $\tkqq=\tkq$, 
$\norm{\tkq}=\norm{T_0}$.

Fix $\delta>0$. Since $T_\ka$ is compact, there is a finite rank operator $F$
with $\norm{\Delta}<\delta$, where $\Delta=T_\ka-F$.
Let $F_\eps$ and $\Delta_\eps$ denote the operators obtained by
multiplying the kernels of $F$ and $\Delta$ by $f_\eps(x)^{1/2}f_\eps(y)^{1/2}$.
Since $f_\eps\le 1$ holds pointwise, we have
\[
 \norm{\Delta_\eps}\le\norm{\Delta} < \delta.
\]
For any $g\in L^2$ the pointwise product $f_\eps g$ converges to $f_0g$ in $L^2$. Since
$F$ has finite rank, it follows that $\norm{F_\eps-F_0}\to 0$, and hence
that
\[
\limsup_{\eps\to 0} \norm{T_\eps-T_0}
 \le \limsup_{\eps\to 0} \norm{F_\eps-F_0} +\delta = \delta.
\]
Since $\delta>0$ was arbitrary, we have $\norm{T_\eps-T_0}\to 0$,
and in particular $\norm{\tkqq}=\norm{T_\eps}\to
\norm{T_0}=\norm{\tkq}<1$.
Hence, there exists $\eps>0$ such that $\norm{\tkqq}<1$.
But then \eqref{need} holds, because
$ \innprod{\tkqq^j1,1}_{\muqq}\le\norm{\tkqq}^j$.
\end{proof}

\begin{remark}
Chayes and Smith~\cite{ChayesSmith} have recently proved a result related to
\refT{Tbounded}(i) or \refT{Tiid}, for the special case where the type
space $\sss$ is finite. Their model has a fixed number of vertices of each type, which
makes essentially no difference in this finite-type case.
Chayes and Smith consider
(in effect) the number of ordered pairs $(v,w)$ of vertices with
$v$ of type $i$, $w$ of type $j$, and $v$ and $w$ in the same
component, normalized  by dividing by $n$,
showing convergence to the relevant branching process quantity.
These numbers sum to give the susceptibility, so such a result is more
refined than the corresponding result for the susceptibility itself.

In our setting, the analogue is to fix arbitrary measurable
subsets $S$ and $T$ of the type space, and consider $\sus_{S,T}(G_n)$,
which is $1/n$ times the number of pairs $(v,w)$ in the same component with the type
of $v$ lying in $S$ and that of $w$ in $T$. The corresponding
branching process quantity is just $\sus_{S,T}(\ka)$, i.e., the integral
over $x\in S$ of the expected number of particles in $\bpk(x)$
with types in $T$. In analogy with \refT{TBP1}, in the subcritical
case this quantity may be
expressed as $\sus_{S,T}(\kk)=\innprodmu{(I-\tk)\qw1_S,1_T}<\infty$.
It is not hard to see
that the proof of \refT{Tiid} in fact shows that
\begin{equation}\label{sST}
 \sus_{S,T}(G_n)\pto \sus_{S,T}(\ka),
\end{equation}
where $G_n=\gnkx$ is defined on an \iid{} vertex space.
The key point is that, in the light of \refT{Tlower} and its proof, it suffices
to prove a convergence result for the contribution to $\sus_{S,T}(G_n)$
from components of a fixed size $k$. For all the models we consider here,
this may be proved by adapting the methods used to prove
convergence of $N_k(G_n)/n$; we omit the details.
Once we have such convergence, we also obtain the analogue of \eqref{sST}
for $\susq$,
so all our results in this section may be extended in this way,
with the proviso that when considering $\gnkx$ with a general vertex
space $\vxs$ as in~\cite{kernels}, we must assume
that $S$ and $T$ are $\mu$-continuity sets.
\end{remark}

\begin{remark}
We believe that all the results in this section extend, with suitable
modifications, to the random graphs with clustering
introduced by~\citet{clustering}, and generalized (to a form analogous
to $G(A_n)$) in~\cite{cutsub}; these may be seen as the simple
graphs obtained from an appropriate random hypergraph by replacing
each hyperedge by a complete graph on its vertex set. Note that in this case
the appropriate limiting object is a hyperkernel (for the
defintions see~\cite{clustering}), and the corresponding branching process
is now a (multi-type, of course) compound Poisson one.

A key observation is that in such a graph, which is the union of certain
complete graphs, two vertices are in the same component if and only if they
are joined by a path which uses at most one edge from each of these
complete graphs. Roughly
speaking, this means that we need consider only the individual edge probabilities,
and not their correlations, and then arguments such as the proof of \refT{Tiid}
and (at least the first part of) \refT{Tbounded}
go through with little change. It also tells us that the susceptibility of a hyperkernel
is simply that of the corresponding edge kernel; this is no surprise, since
for the expected total size of the branching process all that matters
is (informally) the expected number of type $y$ children of each type $x$
individual, not the details of the distribution. This does not extend
to the modified susceptibility $\susq$, since this depends
on the (type-dependent) survival probability $\rho(\ka;x)$, which
certainly is sensitive to the details of the offspring distribution.

Adapting the proof of \refT{Tsup} needs more work, but we believe it should
be possible. Most of the time, one can work with bounded hyperkernels, where
not only are the individual (hyper)matrix entries uniformly bounded, but there
is a maximum edge cardinality. Taking the $r$-uniform case for simplicity,
one needs to show that the number of $(r-1)$-tuples of
vertices in the giant component in some subset of $\sss^{r-1}$ 
is typically close to what it should be,
since, in the proof of \refL{sc}, the sets $A_{y,i}$ should (presumably)
be replaced by corresponding subsets of $\sss^{r-1}$. 
For strong concentration, one argues as here but using
the appropriate stability result from~\cite{cutsub} in place
of \refT{stab}. Needless to say, since
we have not checked the details, there is always the possibility of unseen
complications!
\end{remark}

\section{Behaviour near the threshold}\label{Sthreshold}

In this section we consider the behaviour of $\sus$ and $\susq$ for a
family $\gl\kk$ of kernels, with $\kk$ fixed and $\gl$ ranging from 0
to $\infty$. Since $\norm{\tlk}=\gl\norm\tk$, 
then, as discussed in \cite{kernels},
$\gl\kk$ is
subcritical, critical and supercritical for $\gl<\glc$,
$\gl=\glc$
and $\gl>\glc$, respectively, where $\glc=\norm\tk\qw$.
Note that if $\norm\tk<\infty$, then $\glc>0$, while if
$\norm\tk=\infty$, then $\glc=0$, so $\gl\kk$ is supercritical for any
$\gl>0$. 

Note also that \refT{Teq} provides an alternative way of finding
$\glc$ (and thus $\norm\tk$): we can try to solve the integral
equation $f=1+\tlk f = 1+\gl\tk f$ and see whether
there exists any integrable positive solution.
This tells us whether $\sus(\gl\kk)$ is finite; since
(by Theorems~\refand{TBP1}{Tcritical}) the susceptibility
is finite in the subcritical case and infinite in the supercritical
case, this information determines $\glc$.
The advantage 
of this approach over attempting to solve \eqref{phik} itself
is that the equation is linear; this is one of the main motivations
for studying $\sus$. (Another is that it tends to evolve very simply 
in time in suitably parameterized models.)

In the subcritical case, 
$\gl<\glc$, we have the following  simple result.
(When we say that a function $f$ defined on the reals is
\emph{analytic}
at a point $x$, we mean that 
there is a neighbourhood of $x$ in which $f$
is given by the sum of a convergent power series;
equivalently, 
$f$ extends to a complex analytic
function in a complex neighbourhood of $x$.)

\begin{theorem}
  \label{TB}
Let $\ka$ be a kernel. Then
$\gl\mapsto\sus(\glk)=\susq(\glk)$ is an increasing, analytic
function on $(0,\glc)$, with a singularity at $\glc$.
Furthermore, 
$\sus(\glk)\upto\sus(\glck)=\susq(\glck)\le\infty$ as $\gl\upto\glc$,
and $\sus(\glk;x)\upto \sus(\glck;x)$ pointwise. 
\end{theorem}

\begin{proof}
  By \eqref{bp1b},
\begin{equation}\label{tb}
\sus(\glk)= \musssqw\sumj\innprod{\tk^j1,1}\,\gl^j,
\end{equation}
which converges for $0<\gl<\glc$ by \refT{TBP1}. 
Hence, $\sus(\glk)$ is increasing and analytic on $(0,\glc)$.
Moreover, by \refT{Tcritical}(ii), the sum in \eqref{tb} diverges for
$\gl>\glc$; hence the radius of convergence of this power series is
$\glc$. Since the coefficients are non-negative, this implies that
$\sus(\glk)$ is not analytic at $\glc$.

Finally, $\sus(\glk)\upto\sus(\glck)$ as $\gl\upto\glc$ by
\eqref{tb} and monotone convergence.
Similarly, $\sus(\glk;x)\upto\sus(\glck;x)$ by \eqref{bp1a}
and monotone convergence.
\end{proof}

We shall see in \refSS{SSCHKNS} that it is possible to have 
$\sus(\glck)<\infty$. As we shall now show, 
if $\tk$ is compact, then $\sus(\glck)=\infty$, and the
critical exponent of $\sus$ is $-1$, as $\gl\upto\glc$.

\begin{theorem}\label{Tsub}
  Suppose that $\tk$ is compact (for example, that
  $\int\kk^2<\infty$).
Then for some constant $a$, $0<a\le1$, we have
\begin{equation*}
  \sus(\glk)=\susq(\glk)=\frac{a\glc}{\glc-\gl}+O(1),
\qquad 0<\gl<\glc, 
\end{equation*}
and $\sus(\glck)=\susq(\glck)=\infty$.

If, in addition, $\kk$ is irreducible, then 
$a=\bigpar{\ints\psi}^2/\ints\psi^2$, where $\psi$ is any
  non-negative eigenfunction of $\tk$.
\end{theorem}

\begin{proof}
  Since a compact operator is bounded, $\glc>0$. 
We may assume that $\musss=1$ by \refR{Renorm}. Furthermore, 
we may replace $\kk$
  by $\glc\kk$ and may thus assume, for convenience, that $\norm\tk=1$
  and $\glc=1$.

Let $E_1$ be the eigenspace \set{f\in L^2(\mu):\tk f=f} of $\tk$, and
$P_1$ the orthogonal projection onto $E_1$.
Since $\tk$ is compact and self-adjoint, $E_1$ and its orthogonal
complement are invariant, 1 does not belong to the spectrum of $\tk$
restricted to $E_1^\perp$, and, for $\gl<1$,
$\norm{(I-\gl\tk)\qw(I-P_1)}=O(1)$, while 
$(I-\gl\tk)\qw P_1=(1-\gl)\qw P_1$.
Consequently, by \refT{TBP1},
\begin{equation*}
  \sus(\glk)=(1-\gl)\qw\innprod{P_11,1}+O(1).
\end{equation*}
Let $a\=\innprod{P_11,1}=\normll{P_11}^2\ge0$; then $a\le\normll{1}^2=1$,
so $0\le a \le 1$.
If $a=0$, then $P_11=0$, so the constant function $1$ is orthogonal to
$E_1$. But this contradicts the fact that $E_1$ always contains a
non-zero eigenfunction $\psi\ge0$, see the proof of \refT{Tcritical}
and \cite[Lemma 5.15]{kernels}.
Hence, $a>0$.

The fact that $\sus(\glck)=\infty$ now follows from \refT{TB}.

Furthermore, if $\kk$ is irreducible, then $E_1$ is one-dimensional,
see again \cite[Lemma 5.15 and its proof]{kernels},
so $P_1f=\normll{\psi}^{-2}\innprod{f,\psi}\psi$, and the formula for
$a$ follows, noting that every non-negative eigenfunction is a
multiple of this $\psi$.
\end{proof}

In the supercritical case, only $\susq$ is of interest. 
If we allow reducible $\kk$, we can have several singularities, coming
from different parts of the type space, see \refE{E2}. We therefore
assume that $\kk$ is irreducible. Even in that case, it is possible
that the dual kernel $\kkq$ is critical, 
see \cite[Example 12.4]{kernels}; in this example
it is not hard to check that  $\susq(\kk)$ is infinite.

We conjecture that when $\kk$ is irreducible,
$\susq(\glk)$ is analytic for all $\gl\neq\glc$ 
under very weak conditions, but we have only been able to show this
under the rather stringent condition \eqref{t5a} below.
(See also the examples in \refS{Sex}.)
Under this condition, we can also show that the
behaviour of $\susq$ is symmetric at $\glc$ to the first order: 
the asymptotic behaviour is the same at the subcritical and
supercritical sides. As seen in Examples \refand{Em}{Er}, this does not 
hold for all $\kk$, even if we assume the \HS{} condition
$\int\kk^2<\infty$. (Furthermore, we shall see in Sections
\refand{SSER}{SSrank1} that 
the second order terms generally differ between the two sides.)

\begin{theorem}\label{Tanalytic}
Suppose that $\kk$ is irreducible, and that
  \begin{equation}\label{t5a}
\sup_x \int_\sss \kk(x,y)^2\dd\mu(y) <\infty.
  \end{equation}
\begin{romenumerate}
\item The function $\gl\mapsto\susq(\glk)$ is analytic
except at
  $\glc\=\norm{\tk}\qw$.
\item As $\gl\to\glc$,
\begin{equation*}
  \susq(\glk)=\frac{b\glc}{|\gl-\glc|}+O(1),
\end{equation*}
with 
$b=\bigpar{\ints\psi}^2/\ints\psi^2>0$, 
where $\psi$ is any
non-negative eigenfunction of $\tk$.
\end{romenumerate}
\end{theorem}

\begin{proof}
The subcritical case $\gl<\glc$ follows from \refT{Tsub}, so we assume
$\gl>\glc$. (Note that \eqref{t5a} implies that $\tk$ is \HS{} and
thus compact.)
We may further assume that $\musss=1$.

\pfitem{i} 
Let $\glo>\glc$.
By \cite[Section 15]{kernels}, there exists 
an analytic function $z\mapsto\extrho_z$ defined in a complex
neighbourhood $U$
of $\glo$ and with values in the Banach space $L^2(\mu)$ such that
$\extrho_z=\rho_{z\kk}$ when $z$ is real, and \eqref{phik} extends to 
\begin{equation}\label{phiz}
  \extrho_z = 1-e^{-z\tk \extrho_z}.
\end{equation}
We may further (by shrinking $U$) assume that $\normll{\extrhoz}$ is
bounded in $U$. Then, by \eqref{t5a}
and Cauchy--Schwartz, 
$\normoo{\tk(\extrhoz)}=O(1)$ 
in $U$, and thus, by \eqref{phiz}, $|1-\extrhoz|$ is bounded above and
below, uniformly for  $z\in U$.
In particular, for every $\glk$ with real $\gl\in U$,
$L^2(\muq)=L^2(\mu)$, with uniformly equivalent norms. We can
therefore regard $\tlkq$ as an operator in $L^2(\mu)$.

We define, for $z\in U$, $\tqz f\=z\tk((1-\extrhoz)f)$; thus
$\tql=\tlkq$ for real $\gl\in U$ by \eqref{tkq}. Note that
$z\mapsto\tqz$ is an analytic map of $U$ into the Banach space of bounded
operators on $L^2(\mu)$.

By \refT{TBP1}, $I-\Txq{\glo\kk}$ is invertible. By continuity, 
we may assume that $I-\tqz$ is invertible in $U$. 
Then $f(z)\=\innprodmu{(I-\tqz)\qw1,1-\extrhoz}$ is an analytic
function in $U$, and $f(\gl)=\susq(\glk)$ for real $\gl\in U$ by
\refT{TBP1}(ii). Hence $\susq(\glk)$ is analytic at $\glo$.

\pfitem{ii}
We use a result from perturbation theory, for convenience stated as
\refL{Lpert} below in a form adapted to our purposes; 
see \cite[Section VII.6]{Dunford-Schwartz}
or \cite{Kato} for similar arguments and many related results.

We may rescale and assume that $\glc=\norm{\tk}=1$, \ie, $\kk$ is critical.

It will be convenient to use the fixed Hilbert space $L^2(\mu)$ rather
than $L^2(\muq)$; recall that $\muq$ depends on $\gl$.
Define a self-adjoint operator $\ttl$ in $L^2(\mu)$ by
\begin{equation}\label{l4a}
  \ttl f \= (1-\rholk)\qq\gl\tk(f(1-\rholk)\qq),
\end{equation}
and note that if $\ul$ is the unitary mapping $f\mapsto\uqq f$ of
$L^2(\muq)$ onto $L^2(\mu)$, then $\ttl=\ul\tlkq\ul\qw$ by
\eqref{tkq}. Hence, $\ttl$ in $L^2(\mu)$ is unitarily equivalent to
$\tlkq$ in $L^2(\muq)$. Further, by \refT{TBP1}(ii),
\begin{equation}\label{l4}
  \susq(\glk)
=
\innprodmuq{(I-\tlkq)\qw1,1}
=
\innprodmu{(I-\ttl)\qw\ul1,\ul1}.
\end{equation}

Note that $\rho_\kk=0$, and thus $\tti=\tk$, which has a 
simple eigenvalue 1, with a positive eigenfunction $\psi$
\cite[Lemma 5.15]{kernels}, and all other eigenvalues strictly
smaller. We may assume that $\normll\psi=1$.

We apply \refL{Lpert} with $T=\tti$ and $T'=\ttl$, with $\gl=1+\eps$
for small $\eps>0$.
By \cite[Section 15]{kernels}, $\normoo{\rholk}=O(\eps)$, and more
precisely, $\rholk=a_\eps \psi+\rho_\eps^*$ with
$\normll{\rho_\eps^*}=O(\eps^2)$ and 
\begin{equation}\label{l5}
  a_\eps=\frac2{\ints\psi^3\dd\mu}\eps+O(\eps^2).
\end{equation}
It follows (recalling that $\psi$ is bounded because
$\psi=\tk\psi$ and \eqref{t5a}) that
$\uqq\psi=\psi-\tfrac12a_\eps\psi^2+\reps$, with $\normll{\reps}=O(\eps^2)$.
Consequently, 
\eqref{l4a} implies that
$\norm{\ttl-\tti}=O(\eps)$ and,
using 
$\innprod{\tk\psi^2,\psi}=\innprod{\psi^2,\tk\psi}=\innprod{\psi^2,\psi}
=\ints\psi^3\dd\mu$ and \eqref{l5},
\begin{equation*}
  \begin{split}
  \innprod{\ttl\psi,\psi}
&=
\gl\biginnprod{\tk\bigpar{\uqq\psi},\uqq\psi}
\\&
=\gl\bigpar{\innprod{\tk\psi,\psi}-\tfrac12a_\eps\innprod{\tk\psi,\psi^2}
 -\tfrac12a_\eps\innprod{\tk\psi^2,\psi}
 +O(\eps^2)}
\\&
=(1+\eps)(1-2\eps+O(\eps^2))
\\&
=1-\eps+O(\eps^2).	
  \end{split}
\end{equation*}
Further, $\ul1=\uqq=1+O(\eps)$. Hence, \eqref{l4} and \eqref{lpert}
yield
\begin{equation*}
  \susq((1+\eps)\kk)=\frac{\innprod{1,\psi}^2+O(\eps)}{\eps+O(\eps^2)}+O(1)
=
\frac{\innprod{1,\psi}^2}{\eps}+O(1),
\end{equation*}
which is the desired result.
\end{proof}

\begin{lemma}\label{Lpert}
  Let $T$ be a compact self-adjoint operator in a Hilbert space $H$, such
  that $T$ has a largest eigenvalue $1$ that is simple, with a
  corresponding normalized eigenvector $\psi$.
Then there exists $\eta>0$ such that if $T'$ is any self-adjoint operator
  with $\norm {T'-T}<\eta$ such that $I-T'$ is  invertible, then
  \begin{equation}\label{lpert}
\innprod{(I-T')\qw f,g} 
= \frac{\innprod{f,\psi}\innprod{\psi,g}+\ot}
{1-\innprod{T'\psi,\psi}+\ott}
+O(1).
  \end{equation}
for any $f,g\in H$ with $\norm f,\norm g\le1$.
\end{lemma}

\begin{proof}
The spectrum $\sigma(T)\subset(-\infty,1-\gd]\cup\set1$ for some
$\gd>0$.
Let $\gam$ be the circle \set{z:|z-1|=\gd/2}.
Then, as is well known,  the spectral projection
\begin{equation}\label{p0}
  P_0\=\frac1{2\pi\ii}\oint_\gam (zI-T)\qw\dd z
\end{equation}
is the orthogonal projection onto the one-dimensional eigenspace
spanned by $\psi$. 
Let $A=T'-T$.
If $A$ is any self-adjoint operator with 
$\norm A\le \eta$, for some sufficiently small $\eta>0$, then
$zI-T-A$ is invertible for $z\in\gam$, and we define
\begin{equation}\label{pa}
  P_A\=\frac1{2\pi\ii}\oint_\gam (zI-T-A)\qw\dd z.
\end{equation}
Thus $P_A$ is the spectral projection for $T+A$ associated to the
interior of $\gam$. It follows from \eqref{p0} and \eqref{pa} that
$\norm{P_A-P_0}=O(\norm A)$, so if $\eta$ is small enough,
$\norm{P_A-P_0}<1$, and it follows \cite[Lemma VII.6.7]{Dunford-Schwartz}
that $P_A$ too has rank 1; this must be the orthogonal projection onto
a one-dimensional space spanned by an eigenfunction $\psia$ of
$T+A$ with eigenvalue $\gla$, with $|\gla-1|<\gd/2$. Moreover,
if $\gla\neq1$, then
since all other eigenvalues of $T+A$ then lie outside $\gam$,
\begin{equation}
  \label{l1}
(I-(T+A))\qw=(1-\gla)\qw P_A + R_A, 
\end{equation}
with $\norm{R_A}\le 2/\gd=O(1)$.

Since $\norm{P_A \psi-\psi}=\norm{(P_A-P_0)\psi}=O(\norm A)$,
$P_A\psi\neq0$ (provided $\eta$ is small enough), and thus we can
take $\psia=P_A\psi$. 
Hence $\norm{\psia-\psi}=\norm{P_A\psi-\psi}=O(\norm A)$ and
\begin{align*}
  \innprod{\psia,\psi} &= \innprod{\psi,\psi}+\oa = 1+\oa,
\\
  \innprod{T \psia,\psi} &= \innprod{\psia,T\psi}
= \innprod{\psia,\psi} = 1+\oa,
\\
  \innprod{A \psia,\psi} &= \innprod{A\psi,\psi}+\oaa,
\end{align*}
and thus
\begin{equation}\label{l2}
  \gla= \frac{\innprod{(T+A) \psia,\psi}}{\innprod{\psia,\psi}}
=1+\frac{\innprod{A \psia,\psi}}{\innprod{\psia,\psi}}
=1+\innprod{A\psi,\psi}+\oaa.
\end{equation}
The result follows from \eqref{l1} and \eqref{l2}, using $P_0
f=\innprod{f,\psi}\psi$. 
\end{proof}

\section{Examples}\label{Sex}

In this section we give several examples illustrating the results
above and their limits.
We sometimes drop $\kk$ from the notation; we let $\rhok$ denote the
function $\rhok(x)=\rhok(\kk;x)$. (But we continue to denote the
number $\ints\rhok\dd\mu$ by $\rhokkk$, in order to distinguish it
from the function $\rhok$.)

Note first that the probabilities $\rhok(x)$ can in principle be
calculated by recursion and integration. The number of children of an
individual of type $x$ in the branching process is Poisson with mean
$\int\kk(x,y)\dd\mu(y)=\tk1(x)$, and thus (in somewhat informal language)
\begin{equation}\label{rho1}
  \rho_1(x)=\P(\text{$x$ has no child})
=e^{-\tk1(x)}.
\end{equation}

Next, $|\bpk(x)|=2$ if and only if $x$ has a single child, which is
childless. Hence, by conditioning on the offspring of $x$,
\begin{equation}\label{rho2}
  \begin{split}
  \rho_2(x)
&=e^{-\tk1(x)}\ints\kk(x,y)\P(|\bpk(y)|=1)\dd\mu(y)
=e^{-\tk1(x)}\tk(\rho_1)(x)
\\&
=\rho_1(x)\tk(\rho_1)(x).	
  \end{split}
\end{equation}

Similarly, considering the two ways to get $|\bpk(x)|=3$, 
\begin{equation}\label{rho3}
  \begin{split}
  \rho_3(x)
&=e^{-\tk1(x)}\ints\kk(x,y)\rho_2(y)\dd\mu(y)
\\&\qquad
+e^{-\tk1(x)}\frac12\ints\kk(x,y)\rho_1(y)\dd\mu(y)
\ints\kk(x,z)\rho_1(z)\dd\mu(z)
\\&
=\rho_1(x)\tk(\rho_2)(x)+\tfrac12\rho_1(x)\bigpar{\tk(\rho_1)(x)}^2,	
  \end{split}
\end{equation}
and the three ways to get $|\bpk(x)|=4$, 
\begin{equation}\label{rho4}
  \rho_4=\rho_1T(\rho_3)+\rho_1T(\rho_1)T(\rho_2)+\frac16\rho_1(T \rho_1)^3,
\end{equation}
and so on. In general, for $\rho_k$, $k\ge2$, we get one term
$\rho_1\prod_jT(\rho_j)^{m_j} /m_j!$
for each partition $1^{m_1}2^{m_2}\dotsm$ of $k-1$.

The numbers $\rhokkk$ are then obtained by integration.

Alternatively, a similar recursion can be given for the probability
that $\bpk(x)$ has the shape of a given tree; this can then be summed
over all trees of a given size.

\subsection{The \ER{} case}
  \label{SSER}
Let $\sss$ consist of a single point, with $\musss=1$. Thus, $\kk$
is a positive number.
(More generally, a constant $\kk$ on any probability space
$(\sss,\mu)$ yields the same results.)
We keep to more traditional notation by letting $\kk=\gl>0$; then
$\gnk=\gnp$ with $p=\gl/n$. See \cite[Example 4.1]{kernels}.

Since $\tk$ is just multiplication by $\gl$, $\norm{\tk}=\gl$, and,
as is well-known, $\kk$ is subcritical if $\gl<1$, critical if
$\gl=1$, and supercritical if $\gl>1$. 

In the subcritical case, by \eqref{bp1b} or \refT{TBP1}(i),
\begin{equation}\label{suser}
  \sus(\kk)=\frac1{1-\gl},
\qquad \gl<1.
\end{equation}
\refT{Tbounded} or
\refT{Tiid} shows that 
$\sus(\gnln)\pto(1-\gl)\qw$ for every constant $\gl<1$. (This and
more detailed results are shown by \citet{SJ218} by another
method. See also \citet[Section 2.2]{Durrett} for the expectation
$\E\sus(\gnln)$.)

Similarly, if $\gl\ge1$ then $\sus(\gnln)\pto\sus(\kk)=\infty$ by
\refT{Tcritical} and any of Theorems \ref{Tinfty}, \ref{Tbounded} or
\ref{Tiid}.

For $\susq$, we have the same results for $\gl\le1$. In the
supercritical case $\gl>1$, $\tkq$ is multiplication by
$\gl(1-\rho(\gl))<1$, where $1-\rho(\gl)=\exp(-\gl\rho(\gl))$ by
\eqref{phik}.
Hence, by Theorems \refand{Tbounded}{TBP1}, or \eqref{bp1d},
for $\gl>1$,
\begin{equation}\label{susqer}
  \susq(\gnln)\pto\susq(\kk)=\frac{\muq(\sss)}{1-\gl(1-\rho(\gl))}
=\frac{1-\rho(\gl)}{1-\gl(1-\rho(\gl))}.
\end{equation}

More generally, \refT{Tbounded} shows that $\susq(G(n,\gl_n/n))\pto\susq(\gl)$
for every sequence $\gl_n\to\gl>0$.

For $\gl=1+\eps$, $\eps>0$, we have the Taylor expansion
\begin{align}
 \rho(1+\eps)&=2\eps-\frac83\eps^2+\frac{28}9\eps^3-\frac{464}{135}\eps^4+\dots
\intertext{and thus}
\label{superer}
 \susq(1+\eps)&={\eps}\qw-\frac43+\frac43\eps-\frac{176}{135}\eps^2+\dots
\end{align}
Combining \eqref{suser} and \eqref{superer}, we see that, 
as shown by \refT{Tanalytic}, $\susq(\gl)\sim1/|\gl-1|$ for
$\gl$ on both sides of 1, but the second order terms are different for
$\gl\upto1$ and $\gl\downto1$.

We can also obtain $\sus(\gl)$ and $\susq(\gl)$ from $\rhok$ and the
formulae \eqref{suskk} and \eqref{suskkq}. 
In this case, $\bpk$ is an ordinary, single-type, Galton--Watson
process with Poisson distributed offspring, and
it is well-known, see \eg{}
\cite{Borel,Otter,Tanner,Dwass,Takacs:ballots,Pitman:enum}, that
$|\bpk|$ has a Borel distribution (degenerate if $\gl>1$), \ie,
\begin{equation}
  \rhok(\kk)=\rhok(x)=\frac{k^{k-1}}{k!}\gl^{k-1}e^{-k\gl},
\qquad k\ge1.
\end{equation}
Consequently, if $\cT(z)\=\sum_{k=1}^\infty \frac{k^{k-1}}{k!}z^{k}$ is
the tree function, then 
\begin{equation}\label{er2}
  \rho(\kk)=1-\sum_{1\le k<\infty} \rhok(\kk)=1-\frac{\cT(\gl e^{-\gl})}{\gl}
\end{equation}
and, using the well-known identity $z\cT'(z) = \cT(z)/(1-\cT(z))$, see \eg{}
\cite{SJ97}, 
\begin{equation}\label{er3}
  \susq(\kk)=\sum_{1\le k<\infty} k\rhok(\kk)
=\sumk \frac{k^{k}}{k!}\gl^{k-1}e^{-k\gl}
=\gl\qw\frac{\cT(\gl e^{-\gl})}{1-\cT(\gl e^{-\gl})}.
\end{equation}

In the subcritical case, when $\gl<1$, we have $\cT(\gl e^{-\gl})=\gl$, and we recover
\eqref{suser}. 
In general, \eqref{er2} and \eqref{er3} yield \eqref{susqer}.

\begin{remark}
  Consider the random graph $\gnm$ with a given number $m$ of
  edges. In the subcritical case $m\sim \gl n/2$ with $0<\gl<1$, we
  obtain $\sus(\gnm)\pto\sus(\kk)=1/(1-\gl)$ by comparison with $\gnp$
  with $p=\gl_n/n$ for $\gl_n=2m/n\pm n^{-1/3}$, say, using \refL{Lmon}.
  In the supercritical case $\gl>1$,
one can use standard results on the numbers 
  of vertices and edges in the giant component; conditioning on the giant
 component assuming typical values, the rest of the graph is essentially a subcritical
 instance of $\gnm$ with different parameters; this may be compared
  with $\gnp$ as above. 
Consequently, for
$m\sim \gl n/2$ with $\gl>1$,
  $\susq(\gnm)\pto\susq(\kk)$, where $\susq(\kk)$ is given by
  \eqref{susqer} and \eqref{er3}, just as for
  $\gnp$ with $p=\gl/n$.
\end{remark}

\subsection{The rank 1 case}\label{SSrank1} 
Suppose that $\kk(x,y)=\psi(x)\psi(y)$ for some positive integrable function
$\psi$ on $\sss$. This is the \emph{rank 1 case} studied in
\cite[Section 16.4]{kernels}; note that $\tk$ is the rank 1 operator
$f\mapsto\innprod{f,\psi}\psi$, 
with $\psi$ as eigenfunction,
provided $\psi\in L^2(\mu)$.

We assume, for simplicity, that $\musss=1$. 
As in \refS{Sthreshold} we consider the family
of kernels $\gl\kk$, $\gl>0$. In this case,
$\norm\tk=\normll{\psi}^2=\ints\psi^2$, and thus
$\glc=\normll{\psi}\qww$.

In the subcritical case, $\gl<\glc=\lrpar{\int\psi^2}\qw$, which
entails $\ints\psi^2<\infty$, we have by induction
\begin{equation*}
  \tlk^j1(x) = \gl^j\Bigpar{\ints\psi^2\dd\mu}^{j-1}\ints\psi\dd\mu 
\cdot\psi(x),
\qquad j\ge1,
\end{equation*}
and thus by \eqref{bp1b} (or by solving \eqref{em})
\begin{equation}\label{bsub}
  \begin{split}
  \sus(\glk)=\susq(\glk)
&
=1+\frac{\gl\lrpar{\int\psi}^2}{1-\gl\int\psi^2}
=1+\frac{\gl\lrpar{\int\psi}^2}{1-\gl/\glc}
\\&
=\frac{\lrpar{\int\psi}^2/\int\psi^2}{1-\gl/\glc}
+1-\frac{\lrpar{\int\psi}^2}{\int\psi^2}.	
  \end{split}
\end{equation}
In particular, this verifies the formula in \refT{Tsub}.

In the supercritical case, we first note that the equation \eqref{phik}
for $\rho=\rho_{\glk}$ becomes
\begin{equation}
  \label{b1}
\rho=1-e^{-\gl\tk\rho}
=1-e^{-\gl\innprod{\rho,\psi}\psi}.
\end{equation}
We define $\xi\in\ooop$ by $\xi\=\gl\innprod{\rho,\psi}$, and thus have
\begin{equation}
  \label{b2}
\rho=1-e^{-\xi\psi},
\end{equation}
with $\xi$ given by the implicit equation
\begin{equation}
  \label{b3}
\xi
=\gl\ints\rho(x)\psi(x)\dd\mu(x)
=\gl\ints\psi(x)\Bigpar{1-e^{-\xi\psi(x)}}\dd\mu(x).
\end{equation}
(See \cite[Section 16.4]{kernels}, where the notation is somewhat different.)
We know, by results from \cite{kernels}, that \eqref{b1} has a unique positive
solution $\rho$ for every $\gl>\glc$; thus \eqref{b3} has a unique
solution $\xi=\xi(\gl)>0$ for every $\gl>\glc$.

It is easier to use $\xi$ as a parameter; by \eqref{b3} we have
\begin{equation}
  \label{b4}
\gl
=\frac{\xi}{\int\lrpar{1-e^{-\xi\psi}}\psi}.
\end{equation}
The denominator is finite for every $\xi>0$ since $\psi\in L^1$;
moreover, $\int(1-e^{-\xi\psi})\psi<\int\xi\psi^2$, and thus \eqref{b4}
yields $\gl>1/\int\psi^2=\glc$. Consequently,
\eqref{b3} and \eqref{b4} give a bijection between
$\gl\in(\glc,\infty)$ and $\xi\in\ooo$. 
Furthermore, differentiation of \eqref{b4} shows that $\gl=\gl(\xi)$
is differentiable, and it follows easily from 
$\int(1-e^{-\xi\psi})\psi>\int\xi\psi^2e^{-\xi\psi}$ that $\dd\gl/\dd\xi>0$.
Hence, the function $\gl(\xi)$ and its inverse  $\xi(\gl)$
are both strictly increasing and continuous. 
In particular, $\gl\downto\glc\iff\xi\downto0$.
Moreover, the denominator in \eqref{b4} is an analytic function of
\emph{complex} $\xi$ with $\Re\xi>0$; hence 
$\gl(\xi)$ and its inverse  $\xi(\gl)$ are analytic, for $\xi>0$ and
$\gl>\glc$, respectively.

We note also the following equivalent formula, provided $\ints\psi^2<\infty$:
\begin{equation}
  \label{b4x}
\frac1{\glc}-\frac1{\gl}
={\xi\qw}{\ints\lrpar{e^{-\xi\psi}-1+\xi\psi}\psi}.
\end{equation}

By \eqref{tkq} and \eqref{b2},
\begin{equation}
  \label{b5}
\tlkq f
= \tlk\bigpar{(1-\rho)f}
=\gl\innprod{(1-\rho)f,\psi}\psi
=\gl\ints e^{-\xi\psi(x)}\psi(x)f(x)\dd\mu(x)\,\psi.
\end{equation}
Hence $\tlkq$ too is a rank 1 operator, with eigenfunction $\psi$ and 
eigenvalue (take $f=\psi$ in \eqref{b5})
\begin{equation}
  \label{b6}
\gam
=\gl\ints e^{-\xi\psi(x)}\psi(x)^2\dd\mu(x)
=\frac{\xi\int e^{-\xi\psi}\psi^2}
{\int\lrpar{1-e^{-\xi\psi}}\psi}.
\end{equation}
Since $y^2e^{-y}<y(1-e^{-y})$ for $y>0$, it follows that $0<\gam<1$.
(When $\int\psi^2<\infty$, 
this follows also from the general result \cite[Theorem 6.7]{kernels}, \cf{}
\refT{TBP1}.) Hence $I-\tlkq$ is invertible (in, for example, $L^2(\muq)$),
and by \refT{TBP1}(ii),
\begin{equation}\label{b7}
  \susq(\glk;x)
=(1-\rho(x))(I-\tlkq)\qw1(x)
=e^{-\xi\psi(x)}(I-\tlkq)\qw1(x).
\end{equation}
Let us write $g:=(I-\tlkq)\qw1$. Then, by \eqref{b5}, 
$1=(I-\tlkq)g=g-\zeta\psi$, 
with $  \zeta=\gl\ints e^{-\xi\psi}\psi g$. Hence,
$g=1+\zeta\psi$ and, 
using \eqref{b6},
\begin{equation*}
  \zeta=\gl\ints e^{-\xi\psi}\psi g
=\gl\ints e^{-\xi\psi}\psi +\gl\zeta\ints e^{-\xi\psi}\psi^2
=\gl\ints e^{-\xi\psi}\psi +\zeta\gam.
\end{equation*}
Hence, using \eqref{b4} and \eqref{b6}, 
\begin{equation*}
\zeta=\frac{\gl\int e^{-\xi\psi}\psi}{1-\gam}
=\frac{\xi\int e^{-\xi\psi}\psi}
{\int\bigpar{1- e^{-\xi\psi}}\psi-\xi\int e^{-\xi\psi}\psi^2}.
\end{equation*}
Finally, by \eqref{b7},
\begin{equation}\label{b8}
  \begin{split}
\susq
&=\ints\susq(\glk;x)\dd\mu(x)
=\ints e^{-\xi\psi}g
=\ints e^{-\xi\psi}	+\zeta\ints e^{-\xi\psi}\psi
\\&
=\ints e^{-\xi\psi}
+\frac{\xi\lrpar{\int e^{-\xi\psi}\psi }^2}
{\int\bigpar{1- e^{-\xi\psi}(1+\xi\psi)}\psi}.
  \end{split}
\end{equation}

We observe that \eqref{b8} shows that $\susq$ is an analytic function
of $\xi\in\ooop$, and thus of $\gl\in(\glc,\infty)$. (So in the rank 1
case, at least, the condition \eqref{t5a} 
is not required for  \refT{Tanalytic}(i).)

Next, suppose that $\ints\psi^3<\infty$. In this case, we can
differentiate twice under the integral signs in \eqref{b4} and
\eqref{b8} using
dominated convergence (comparing with $\ints\psi^3$), and 
taking Taylor expansions we see that as $\xi\to 0$ we have
\begin{equation}\label{b9}
\gl
=\frac{\xi}
{\xi\int\psi^2-\frac12\xi^2\int\psi^3+o(\xi^2)}
=\glc+ \frac12\xi\frac{\int\psi^3}{\lrpar{\int\psi^2}^2}+o(\xi)
\end{equation}
and
\begin{equation}\label{b10}
\susq
=O(1)
+\frac{\xi\lrpar{\int\psi+O(\xi)}^2}
{\frac12\xi^2\int\psi^3+o(\xi^2)}
\sim
\frac{2\lrpar{\int\psi}^2}{\int\psi^3}\xi\qw
\sim\frac{\lrpar{\int\psi}^2/\lrpar{\int\psi^2}^2}{\gl-\glc}
,
\end{equation}
where we used \eqref{b9} in the last step.

Note that \eqref{bsub} and \eqref{b10} show that the behaviour of
$\susq$ at the critical point $\glc$ is symmetrical to the first order: 
\begin{equation}\label{b11}
\susq(\glk)
\sim\frac{\lrpar{\int\psi}^2/\lrpar{\int\psi^2}^2}{|\gl-\glc|}
= \frac{\lrpar{\int\psi}^2/\int\psi^2}{|\gl/\glc-1|} ,
\qquad \gl\to\glc,
\end{equation}
at least when $\int\psi^3<\infty$.
(This is the same first order asymptotics as given by
\refT{Tanalytic}(ii), but note that the latter applies only when
$\psi$ is bounded, since \eqref{t5a} fails otherwise.)
The second order terms are different on the two sides of $\glc$, though: 
if $\int\psi^4<\infty$,
then carrying the Taylor expansions above one step further leads to
\begin{equation}\label{bsup}
  \begin{split}
\susq(\glk)
&
=\frac{\lrpar{\int\psi}^2/\int\psi^2}{\gl/\glc-1}
+1+\frac{\lrpar{\int\psi}^2}{\int\psi^2}	
-\frac{4\int\psi\int\psi^2}{\int\psi^3}	
+\frac{2\lrpar{\int\psi}^2\int\psi^4}{3\lrpar{\int\psi^3}^2}	
\\&\qquad
+o(1),
\qquad \gl\downto\glc,
  \end{split}
\end{equation}
in contrast to \eqref{bsub} for $\gl<\glc$.

To see what may happen if $\ints\psi^3=\infty$, we look at a few specific
examples. 

\begin{example}\label{Em}
Let $2<q<3$ and take $\sss=[1,\infty)$ with
  $\dd\mu(x)=qx^{-q-1}\dd x$, and take $\psi(x)=x$; note that $\ints
  \psi^p<\infty$ if and only if $p<q$; in particular
  $\ints\psi^2<\infty$ but $\ints\psi^3=\infty$.
By \eqref{b4x}, and standard integration by parts of Gamma integrals,
as $\xi\to0$ we have
\begin{equation*}
  \begin{split}
	\frac1{\glc}-\frac1{\gl}
&=
\xi\qw\int_1^\infty \bigpar{e^{-\xi x}-1+\xi x}qx^{-q}\dd x
=
q\xi^{q-2}\int_\xi^\infty \bigpar{e^{-y}-1+y}y^{-q}\dd y
\\&
\sim q\xi^{q-2}\int_0^\infty \bigpar{e^{-y}-1+y}y^{-q}\dd y
= q\xi^{q-2} \Gamma(1-q),
  \end{split}
\end{equation*}
or $\gl-\glc\sim q\Gamma(1-q)\glc^2\xi^{q-2}$.
Similarly, by another integration by parts,
\begin{equation*}
  \begin{split}
\ints\bigpar{1&- e^{-\xi\psi}(1+\xi\psi)}\psi\dd\mu
=\int_1^\infty \bigpar{1-e^{-\xi x}(1+\xi x)}qx^{-q}\dd x
\\&
=q\xi^{q-1} \int_\xi^\infty \bigpar{1-e^{-y}(1+y)}y^{-q}\dd y
\sim q\xi^{q-1} \int_0^\infty \bigpar{1-e^{-y}(1+y)}y^{-q}\dd y
\\&
= \frac{q\xi^{q-1}}{q-1} \Gamma(3-q)
= q(q-2)\xi^{q-1} \Gamma(1-q),
  \end{split}
\end{equation*}
and thus by \eqref{b8},
\begin{equation*}
  \begin{split}
\susq
\sim\frac{\xi\lrpar{\int\psi}^2}
{q(q-2)\xi^{q-1} \Gamma(1-q)}
\sim\frac{\lrpar{\int\psi}^2\glc^2}
{(q-2)(\gl-\glc)},
\qquad \gl\downto\glc,
  \end{split}
\end{equation*}
which still has power $-1$, but differs by a factor $(q-2)\qw$ from
the subcritical asymptotics in \eqref{bsub} and \refT{Tsub}.
Hence, \eqref{b11} does \emph{not} hold in general without assuming
$\ints\psi^3<\infty$. (Although this integral does not appear in the
formula.)
\end{example}

\begin{example}\label{Er}
We see in \refE{Em} that $\susq$ is relatively large in the
barely supercritical phase when
$\psi$ is only a little more than square integrable. We can pursue
this further by taking the same $\sss$ and $\psi$, and
$\dd\mu(x)=c(\log x+1)^{-q}x^{-3}\dd x$ with $q>1$ and a normalization
constant $c$. Similar calculations using \eqref{b4x} and \eqref{b10} 
(we omit the details) show that, letting 
$c$ denote different positive constants (depending on $q$), as $\xi\to0$ we have
$\gl-\glc\sim c (\log(1/\xi))^{-(q-1)}$ and 
$\susq\sim c (\log(1/\xi))^{q}$, and thus
\begin{equation*}
  \susq(\glk) \sim c (\gl-\glc)^{-q/(q-1)},
\qquad \gl\downto\glc,
\end{equation*}
with an exponent $-q/(q-1)$, which can be any real number in
$(-\infty,-1)$. 

Taking instead
$\dd\mu(x)=c(\log\log x)\qww(\log x)^{-1}x^{-3}\dd x$, $x>3$, we
similarly find 
$\gl-\glc\sim c (\log\log(1/\xi))\qw$ and 
$\susq\sim c (\log(1/\xi))(\log\log(1/\xi))^2$, and thus
\begin{equation*}
  \susq(\glk) =\exp\lrpar{-\frac{c+o(1)}{\gl-\glc}},
\qquad \gl\downto\glc,
\end{equation*}
with an even more dramatic singularity. Of course, this sequence of
examples can be continued to yield towers of exponents.
\end{example}

\subsection{The CHKNS model}\label{SSCHKNS}
Consider the family of kernels $\glk$, $\gl>0$, with
\begin{align}
\kk(x,y)
\={\frac1{x\bmax y}-1}
\label{chkns}
\end{align}
on $\sss=(0,1]$ with Lebesgue measure $\mu$.
We thus have
\begin{equation}
  \label{ctk}
  \begin{split}
\tlk f(x)& = 
\gl\Bigpar{\frac1x-1}\int_0^x f(y)\dd y 
+ \gl\int_x^1\Bigpar{\frac1y-1} f(y)\dd y
\\&=
\frac\gl x\int_0^x f(y)\dd y 
+ \gl\int_x^1\frac{ f(y)}y\dd y
- \gl\int_0^1 f(y)\dd y.	
  \end{split}
\end{equation}

\begin{remark}
  Equivalently, by a change of variable, we could consider 
the kernel $\gl(e^{x\bmin y}-1)$ on $\sss=\ooo$ with
$\dd\mu=e^{-x}\dd x$; we leave it to the reader to reformulate results
in this setting.
\end{remark}

This kernel arises in connection with the CHKNS model of a random
graph introduced by
\citet{CHKNS}.
This graph grows from a single
  vertex; vertices are added one by one, and after each vertex is
  added, an edge is added with probability $\gd\in(0,1)$; the endpoints are
  chosen uniformly among all existing vertices. 
Following Durrett \cite{Durrett:CHKNS,Durrett}, we consider a modification
where at each step a Poisson $\Po(\gd)$ number of edges are added to the
graph, again with endpoints chosen uniformly at random.
As discussed in detail in \cite[Section 16.3]{kernels}, this yields a 
random graph of the type $\gnkxn$ for a graphical sequence of kernels
$(\kk_n)$ with limit $\glk$, where $\gl=2\gd$, 
on a suitable vertex space $\vxs$ (with $\sss$ and $\mu$ as above).

Let us begin by solving \eqref{em}. If $f=\tlk f+1$, then \eqref{ctk}
implies first that $f\in C(0,1)$ and then $f\in C^1(0,1)$. Hence we
can differentiate and find, using \eqref{ctk} again, that
\begin{equation}\label{julie}
  f'(x)=(\tlk f)'(x) = -\frac{\gl}{x^2} \int_0^x f(y)\dd y.
\end{equation}
With $F(x)\=\int_0^x f(y) \dd y$, this yields $F''(x)=-\gl F(x)/x^2$,
with the solution $F(x)=C_1x^{\ga_+} + C_2 x^{\ga_-}$, where
$\ga_\pm$ are the roots of $\ga(\ga-1)=-\gl$, \ie,
$\ga_\pm=\frac12\pm\sqrt{\frac14-\gl}$; if $\gl=1/4$ we have a double
root $\ga_+=\ga_-=1/2$ and the solution is 
$F(x)= C_1 x^{1/2} + C_2 x^{1/2}\log x$. Hence any integrable solution
of \eqref{em} must be of the form $f(x)=C_+x^{\ga_+-1} + C_- x^{\ga_--1}$,
or $f(x)= C_+ x^{-1/2} + C_- x^{-1/2}\log x$ if $\gl=1/4$. Any such
$f$ satisfies \eqref{julie}, and since \eqref{ctk} yields $\tlk f(1)=0$,
it solves \eqref{em} if and only if $f(1)=1$, \ie, if $C_++C_-=1$
($C_+=1$ if $\gl=1/4$).

If $0<\gl<1/4$, then $0<\ga_-<1/2<\ga_+<1$, so the solution
$f(x)=x^{\ga_+-1}$ is in $L^2(0,1)$ and non-negative; by \refC{Ceq}, 
this is the unique non-negative solution in $L^2$, and
\begin{equation}\label{csub}
 \sus(\glk)=\intoi x^{\ga_+-1}\dd x = \frac{1}{\ga_+}
=\frac2{1+\sqrt{1-4\gl}}
=\frac{1-\sqrt{1-4\gl}}{2\gl}.
\end{equation}
(If we are lucky, or with hindsight, we may observe directly that
$x^{\ga_+-1}$  is a solution of \eqref{em} by \eqref{sjw} below,
and apply \refC{Ceq} directly, eliminating most of the analysis above.)

For $\gl<1/4$, we have shown that $\sus(\glk)$ is finite, 
so $\gl\kk$ is subcritical; thus $\glc\ge1/4$. 
Since the \rhs{} in
\eqref{csub} has a singularity at $\gl=1/4$, \refT{TB} shows that
$\glc>1/4$ is impossible, so we conclude that
$\glc=1/4$. (Equivalently, $\norm{\tk}=4$.) 
This critical value for the
CHKNS model has earlier been found by \citet{CHKNS}
by a non-rigorous method, also using \eqref{csub} which they found in
a different way; 
another non-rigorous proof was given by
Dorogovtsev, Mendes and Samukhin~\cite{DMS-anomalous},
and the first rigorous proof was given  by
Durrett \cite{Durrett:CHKNS,Durrett}. See also
Bollob\'as, Janson and Riordan~\cite{SJ163,kernels},
where different methods were used not involving the susceptibility.
The argument above seems to be new.

By \refT{TB}, we can let $\gl\upto\glc$ in \eqref{csub}, and see that
the equation holds for $\gl=\glc=1/4$ too; i.e., $\sus(\glc\kk)=2$.

We see also that in the (sub)critical case $\gl\le1/4$,
$\sus(\glk;x)=x^{\ga_+-1}$.

We have no need for the other solutions of \eqref{em}, but note
that  our analysis shows that for $\gl<\glc$,
the other non-negative, integrable
solutions of \eqref{em} are given by 
$x^{\ga_+-1} + C(x^{\ga_--1}-x^{\ga_+-1})$, with $C>0$.
Similarly, although we have no need for the solutions of \eqref{em} for
$\gl\ge\glc$, let us note that for the critical case $\gl=\glc$, the
argument above shows that there is a minimal non-negative solution
$x^{-1/2}$, which belongs to $L^1$ but not to $L^2$; there are further
solutions $x^{-1/2}-Cx^{-1/2}\log x$, $C>0$. For $\gl>1/4$, the roots
$\ga_\pm$ are complex, and the only real integrable solution to \eqref{em} is
$\frac12(x^{\ga_+-1}+x^{\ga_--1})=\Re x^{\ga_+-1}=
x^{-1/2}\cos\bigpar{\xpar{\gl-\frac14}\qq\log x}$, which oscillates;
thus there is no finite non-negative solution at all.

Before proceeding to $\susq$ in the supercritical case, let us
calculate $\rho_k$ for small $k$. 
We begin by observing, from \eqref{ctk}, that $\tlk1(x)=-\gl\log x$. 
Hence \eqref{rho1} yields
\begin{equation}
  \label{crho1}
\rhox1(\glk;x)=e^{\gl\log x} = x^\gl.
\end{equation}
Further, by \eqref{ctk}, for every non-zero $\gam>-1$,
\begin{equation}\label{sjw}
  \tlk (x^\gam) =\frac{\gl}{\gam(\gam+1)}(1-x^\gam).
\end{equation}
Hence \eqref{rho2} yields
\begin{equation}
  \label{crho2}
\rhox2(\glk;x)=x^\gl\tlk(x^\gl)=\frac{1}{1+\gl}(x^{\gl}-x^{2\gl}).
\end{equation}
Similarly, \eqref{rho3} and \eqref{rho4} yield
\begin{align}
  \label{crho3}
\rhox3(\glk;x)
&=\frac{(2+3\gl)x^{3\gl}-4(1+2\gl)x^{2\gl}+(2+5\gl)x^\gl}
{2(1+\gl)^2(1+2\gl)},
\end{align}
and a formula for $\rhox4(\glk;x)$ that we omit, 
and so on.
By integration we then obtain
\begin{align}
  \rhox1(\glk)&=\frac{1}{1+\gl},
\\
  \rhox2(\glk)&=\frac{\gl}{(1+\gl)^2(1+2\gl)},
\\
  \rhox3(\glk)&=\frac{3\gl^2}{(1+\gl)^3(1+2\gl)(1+3\gl)},
\\
  \rhox4(\glk)&=\frac{2\gl^3(7+15\gl)}{(1+\gl)^4(1+2\gl)^2(1+3\gl)(1+4\gl)}.
\end{align}

It is obvious that each $\rhok(\glk;x)$ is a polynomial in $x^\gl$
with coefficients that are rational functions in $\gl$, with only
factors $1+j\gl$, $j=1,\dots,k$ in the denominator. Hence,
each $\rho(\glk)$ is a rational function of the same type.

There is no obvious general formula for the numbers $\rhok(\glk)$,
but, surprisingly, they satisfy a simple quadratic
recursion, given in the following theorem. This recursion was found by
\citet{CHKNS}, using their recursive
construction of the graph, see also \cite[Chapter 7.1]{Durrett}.
(The argument in \cite{CHKNS} is non-rigorous, but as
pointed out by Durrett~\cite{Durrett:CHKNS,Durrett}, 
it is not hard to make it rigorous.)
We give here a proof that instead uses the branching process, which
gives more detailed information 
about the distribution of the `locations' of the components. 

\begin{theorem}
For the CHKNS kernel \eqref{chkns}, 
$\rhok(\glk)$ satisfies the recursion
\begin{equation}
  \label{crho}
\rhok(\glk)
=\frac{k\gl}{2(1+k\gl)}\sum_{j=1}^{k-1}\rho_{k-j}(\glk)\rho_j(\glk),
\qquad k\ge2,
\end{equation}
with $\rhox1(\glk)=1/(1+\gl)$.
Hence,
for each $k\ge1$,
$\rhok(\glk)$ is
a rational function of  $\gl$, 
with poles only at $-1/j$, $j=1\dots, k$.

Moreover, each function $\rhok(x)=\rho_k(\glk;x)$ is a polynomial in
$x^\gl$, with coefficients that are rational functions of $\gl$,
which can be calculated recursively by 
\begin{equation}
  \label{c1}
x\ddx\rhok(\glk;x)
=k\gl\rhok(\glk;x)-\sum_{j=1}^{k-1}j\gl\rho_{k-j}(\glk)\rho_j(\glk;x),
\qquad k\ge1,
\end{equation}
together with the boundary conditions $\rhox1(\glk;1)=1$ and
$\rhox{k}(\glk;1)=0$, $k\ge 2$. 
\end{theorem}

\begin{proof}
Fix $\gl>0$. To simplify the notation, throughout this
proof we write $\gk$ for the kernel so far denoted $\glk$.
Let $\eps\in(0,1/2)$, say, and 
  let $\bpk'$ be $\bpk$ with all points scaled by the factor
  $(1-\eps)$;
this is the branching process defined by $\sss'\=(0,1-\eps]$,
  $\dd\mu'\=(1-\eps)\qw\dd x$ and 
$\kk'(x,y)\=\gl\bigpar{\frac{1-\eps}{x\bmax y}-1}$.	
In $\bpk'$, the offspring process of an individual of type $x$ has
  intensity
  \begin{equation*}
\kk'(x,y)\dd\mu'(y)=\gl\Bigpar{\frac{1}{x\bmax y}-\frac1{1-\eps}}\dd y
=\gk(x,y)\dd y - \frac{\eps\gl}{1-\eps} \dd y,
\qquad y\le 1-\eps.		
  \end{equation*}
This is less than the intensity in $\bpk$. 
We let $\kk'(x,y)=0$ if $x>1-\eps$ or $y>1-\eps$, and define
$\kk''(x,y)=\gk(x,y)-\kk'(x,y)\ge0$. More precisely, for
$0<x\le1-\eps$ and $0<y\le 1$,
\begin{equation}
  \label{kk''}
\kk''(x,y)=
\begin{cases}
\frac{\eps\gl}{1-\eps}, & 0<y\le1-\eps,\\
\gl\bigpar{\frac1y-1}\le \frac{\eps\gl}{1-\eps}, & 1-\eps<y\le 1.\\  
\end{cases}
\end{equation}
Thus $\bpk(x)$ and $\bpk'(x)$ may be coupled in the natural way so
that $\bpk'(x)\subseteq\bpk(x)$ in the sense that an individual in
$\bpk'(x)$, of type $z$ say, also belongs to $\bpk(x)$, and its
children in $\bpk(x)$ are its children in $\bpk'(x)$ plus %(possibly)
some children born according to an independent Poisson process with
intensity $\kk''(z,y)\dd y$; we call the latter children (if any)
\emph{adopted}.
An adopted child of type $y$ gets children and further descendants
according to a copy of $\bpk(y)$, independent of everything else.
Note that this adoption intensity $\kk''(x,y)$ is independent of
$x\in\sss'$, and that the total adoption intensity is
$\intoi \kk''(x,y)\dd y=\eps\gl+O(\eps^2)$.

Fix $k\ge1$. If $|\bpk(x)|=k$, then either $|\bpk'(x)|=k$ and there
are no adoptions, or $|\bpk'(x)|=j$ for some $j<k$ and there are one or
more adoptions, with a total family size of $k-j$.
If $|\bpk'(x)|=k$, then the probability of some adoption is
$k\eps\gl+O(\eps^2)$, and thus
\begin{equation}
  \P\bigpar{|\bpk(x)|=k \bigm| |\bpk'(x)|=k}=1-k\gl\eps+O(\eps^2).
\end{equation}
Now, suppose that $|\bpk'(x)|=j<k$. 
The probability of two or more adoptions is $O(\eps^2)$. Suppose that
there is a single adoption. If the adopted child has type $y$, the
probability that this leads to an adopted branch of size $k-j$, and
thus to $|\bpk(x)|=k$, is $\rho_{k-j}(\kk;y)$.
By \eqref{kk''}, the adoption intensity $\kk''(z,y)$ is 
independent of $z$ as remarked above, and is almost
uniform on $(0,1]$; it
follows that the probability that $|\bpk(x)|=k$, given 
$|\bpk'(x)|=j$ and that there is a single adoption, by some individual
  of type $z$ in $\bpk'(x)$, equals
  \begin{equation}
\frac{\intoi\kk''(z,y)\rho_{k-j}(\kk;y)\dd y}{\intoi\kk''(z,y)\dd y}
=\intoi\rho_{k-j}(\kk;y)\dd y + O(\eps) =\rho_{k-j}(\kk)+O(\eps).
  \end{equation}
Since the probability of an adoption at all is $j\eps\gl+O(\eps^2)$,
we obtain
\begin{equation}
  \P(|\bpk(x)|=k\mid|\bpk'(x)|=j)=j\gl\rho_{k-j}(\kk)\eps+O(\eps^2).
\end{equation}

Consequently, for every $k\ge1$ and $x\in(0,1-\eps]$,
\begin{equation}
\rhok(\kk;x)
=(1-k\gl\eps)\rhok(\kk';x)
+\sum_{j=1}^{k-1}j\gl\rho_{k-j}(\gk)\rho_j(\kk';x)\eps
+O(\eps^2).
\end{equation}
(The implicit constant in $O$ here and below may depend on $k$ but not
on $x$ or $\eps$.)
Replace $x$ by $(1-\eps)x$ and observe that, by definition, 
$|\bpk'((1-\eps)x)|\eqd|\bpk(x)|$ and thus
$\rho_j(\kk';(1-\eps)x)=\rho_j(\kk;x)$. This yields
\begin{equation}
\rhok(\kk;(1-\eps)x)
=(1-k\gl\eps)\rhok(\kk;x)
+\sum_{j=1}^{k-1}j\gl\rho_{k-j}(\gk)\rho_j(\kk;x)\eps
+O(\eps^2).
\end{equation}
Letting $\eps\downto0$ we see first that $\rhok(\kk;x)$ is Lipschitz
continuous in $(0,1)$, and then that it is differentiable with
\begin{equation}
  \label{c1x}
x\ddx\rhok(\gk;x)
=k\gl\rhok(\gk;x)-\sum_{j=1}^{k-1}j\gl\rho_{k-j}(\gk)\rho_j(\gk;x),
\qquad k\ge1,
\end{equation}
which is \eqref{c1} in the present notation.

For $k=1$, \eqref{c1x} gives $\rhox1(\kk;x)=C x^\gl$, for some constant 
$C$. For $x=1$ we have $\gk(1,y)=0$, so
the branching process $\bpk(x)$ dies immediately, and $\rhox1(\kk;x)=1$.
Thus $\rhox1(\kk;x)=x^\gl$ as shown
in \eqref{crho1}.
For $k\ge2$, we note that $x\rhok(\kk;x)\to0$ as $x\to0$ or $x\to1$,
because $\rhok(\kk;x)\le1-\rho_1(\kk;x)=1-x^\gl$, and thus, integrating by
parts,
\begin{equation*}
  \intoi x\ddx \rhok(\kk;x)
=\bigsqpar{x\rhok(\kk;x)}_0^1 -\intoi\rho_k(\kk;x)\dd x = 0-\rhok(\kk).
\end{equation*}
Hence, integration of \eqref{c1x} yields the recursion formula
\begin{equation}
\label{c2}
(1+k\gl)\rhok(\gk)
=\sum_{j=1}^{k-1}j\gl\rho_{k-j}(\gk)\rho_j(\gk),
\qquad k\ge2.
\end{equation}

Replacing $j$ by $k-j$ in the \rhs{} of \eqref{c2} and summing the two
equations, we find that
\begin{equation}
\label{c2b}
2(1+k\gl)\rhok(\gk)
=\sum_{j=1}^{k-1}(j+k-j)\gl\rho_{k-j}(\gk)\rho_j(\gk),
\qquad k\ge2,
\end{equation}
which is \eqref{crho}.
\end{proof}

The susceptibility $\susq$ was calculated for all $\gl$ by \citet{CHKNS} 
using the recursion formula \eqref{crho}, see also Durrett
\cite{Durrett:CHKNS,Durrett}. We 
repeat their argument for completeness.

Let $G(z)\=\sumk\rhok(\gl\kk)z^k$ be the probability generating
function of $\qbplk$, defined at least for $|z|\le1$. Note that 
in the supercritical case, $\qbplk$ is a defective random variable 
which may be $\infty$; we have
$G(1)=1-\P(\qbplk=\infty)=1-\rho(\glk)$. Further,
$G'(1)=\susq(\glk)\le\infty$.

The recursion \eqref{crho} yields, most easily from the version
\eqref{c2},
\begin{equation}\label{cg}
  G(z)+\gl zG'(z) = \gl zG'(z)G(z)+(1+\gl)\rhox1(\glk)z
=\gl zG'(z)G(z)+z,
\end{equation}
and thus
\begin{equation}
\label{G'}
  G'(z)
=\frac{z-G(z)}{\gl z(1-G(z))},
\qquad |z|<1.
\end{equation}
In the supercritical case, $G(1)<1$, and we can let $z\upto1$ in
\eqref{G'}, yielding $\susq(\glk)=G'(1)=1/\gl$.
(In the subcritical case, l'H\^opital's rule, or differentiation of
\eqref{cg},  yields a quadratic
equation for $G'(1)$, with \eqref{csub} as a solution; this is the
method by which \eqref{csub} was found in \cite{CHKNS}.)

Summarizing, we have rigorously verified the explicit formula by
\citet{CHKNS}: 
\begin{equation}\label{csus}
 \susq(\glk)
=
\begin{cases}
  \frac{1-\sqrt{1-4\gl}}{2\gl}, & \gl\le \frac14,
\\
\frac1{\gl}, & \gl>\frac14.
\end{cases}
\end{equation}
Note that there is a singularity at $\gl=1/4$ with a finite jump from
2 to 4, with infinite derivative on the left side and finite
derivative on the right side. 
It is striking that there is a simple
explicit formula for $\susq(\glk)=G'(1)$, while no formula is known
for $G(1)=1-\rho(\glk)$.
This is presumably related to the fact that $\susq(\glk)$ may be found 
by solving the linear equation \eqref{em}, whereas
$\rho(\glk)$ is related to the non-linear equation \eqref{phik}.
As $\gl=1/4+\eps\downto 1/4$,
$\rho(\glk)$ approaches 0 extremely rapidly, as 
$\exp\bigpar{-(\pi/2\sqrt2)\eps\qqw+O(\log\eps)}$
\cite{DMS-anomalous,kernels}; the behaviour at the singularity is thus
very different for $G(1)$ and $G'(1)$.

Note also that, by \eqref{suskkq},
the discontinuous function $\susq(\glk)$
is the pointwise sum of the analytic functions $k\rhokkk$.

\begin{remark}  
  We can obtain higher moments of the distribution $(\rho_k(\glk))_{k\ge1}$ 
of $\qbplk$ by
repeatedly differentiating the differential equation \eqref{G'}
for its probability generating function and then letting $z\upto1$.
In the supercritical case, this 
yields the moments of $\qbplk\ett{\qbplk<\infty}$
(or, equivalently, the moments of $\qbplk$ conditioned on
$\qbplk<\infty$); it follows that 
all these moments are finite, and
we can obtain explicit formulae for them one by one. For example,
with $\rho=\rho(\glk)$,
\begin{align}
  \E(\qbplk^2 ; |\bplk|<\infty)
&=G''(1)+G'(1) = \frac{1-\rho}{\gl\rho}+\frac1{\gl}
= \frac{1}{\gl\rho},
\label{x2}
\\
  \E(\qbplk^3 ; |\bplk|<\infty)
&=G'''(1)+3G''(1)+G'(1) 
= \frac{2}{\gl^2\rho^2}+ \frac{1}{\gl\rho}.
\label{x3}
\end{align}
It can be 
seen that for each $m\ge 1$,
as $\gl\downto\glc$, and thus
$\rho\to0$, we have 
\begin{equation}\label{moments}
 \E(\qbplk^m;|\bplk|<\infty)\sim c_m \rho^{1-m}
\end{equation}
for some constant $c_m>0$; we do not know any general formula for $c_m$.
For any $\gl>\glc=\tfrac14$ and $a,b>0$,
writing $\hx\=\qbplk\ett{\qbplk<\infty}$, from \eqref{csus} and \eqref{x2}--\eqref{x3} we obtain
\begin{align*}
  \E\Bigpar{\hx^2&;\hx\le \frac a\rho}
\le \frac{a}{\rho}\E\hx
= \frac{a}{\rho}\susq(\glk)
= \frac{a}{\gl\rho},
\\
  \E\Bigpar{\hx^2&;\hx\ge \frac b\rho}
\le \frac{\rho}{b}\E\hx^3
= \frac{2}{b\gl^2\rho}+\frac{1}{b\gl},
\end{align*}
and hence
\[
  \E\Bigpar{\hx^2;\frac a\rho\le\hx\le \frac b\rho}
\ge \frac{1}{\gl\rho}- \frac{a}{\gl\rho}-\frac{2}{b\gl^2\rho}-\frac{1}{b\gl}
= \frac{1}{\gl\rho}\Bigpar{1-a-\frac{2}{b\gl}-\frac{\rho}{b}}.
\]
Choosing, for example, $a=1/4$ and $b=32$, so $b\gl>8$, the
last quantity is at least $1/(3\gl\rho)>1.3/\rho$ if
$\gl$ is close to $\glc$, and thus, for such
$\gl$ at least,
\begin{align*}
  \P\Bigpar{\frac 1{4\rho}\le\qbplk\le \frac{32}{\rho}}
\ge \frac{1.3}{\rho}\Bigparfrac{\rho}{b}^2>\frac{\rho}{1000}.
\end{align*}
Hence, $\qbplk$ may be as large as about $\rho\qw$ with
probability about $\rho$, as suggested by \eqref{moments}.

Note that each $\rho_k(\glk)$ is a
continuous function of $\gl$, so as $\gl\downto\glc$, 
the (defective) distribution of $\qbplk$ converges to the distribution of the
critical $|\bpglck|$, which has mean $\sus(\glc\kk)=2$ and 
$\P(|\bpglck|=k)\sim 2/(k^2\log k)$ as $k\to\infty$, see 
\cite[Section 7.3]{Durrett}.

In the subcritical case, $\rho_k(\glk)$ decreases as a power of $k$,
see \cite[Section 7.3]{Durrett} for details.
\end{remark}

We have so far studied $\sus(\glk)$ and $\susq(\glk)$, or, equivalently,
the cluster size in the branching process $\bplk$. 
Let us now return to the random graphs; we then
have to be careful with the precise definitions.
The Poisson version of the CHKNS model mentioned above can be
described as the random multigraph where the number of edges between vertices
$i$ and $j$ is $\Po(\glij)$ with intensity
$\glij\=\gl(1/(j-1)-1/n)$, for $1\le i<j\le n$, 
independently for all such pairs $i,j$, see
\cite{Durrett:CHKNS,Durrett,kernels}. 
For the moment, let us call this random graph $\gin$. Let
$\giin$ be defined similarly, but with 
$\glij\=\gl(1/j-1/n)$, and let 
$\giiin$ be defined similarly with
$\glij\=\gl(1/j-1/(n+1))$, for $1\le i<j\le n$.
Since multiple edges do not matter for the components, we may as well
consider the corresponding simple graphs with multiple edges
coalesced; then the probability of an edge between $i$ and $j$, $i<j$,
is $\pij:=1-\exp(-\glij)$.
(If, for simplicity, we consider $\gl\le1$ only,
it is easy to see that the results below hold also if we instead
let the edges appear with probabilities $\pij=\glij$; this follows by
the same arguments or by contiguity and 
\cite[Corollary 2.12(iii)]{SJ212}.)

We first consider $\giin$; note that this is exactly (the
Poisson version of)
$\gnxx{\glk}$ with $\kk$ defined in \eqref{chkns} and the vertex space
$\vxs$ given by $\sss=(0,1]$ with $\mu$ Lebesgue measure as above, and
the deterministic sequence 
$\xs=(x_1,\ldots,x_n)$ with $x_i=i/n$.  
Arguing as in the proof of \refT{Tbounded},
summing
over distinct indices only, and using the fact that $\kk$ is
non-increasing in each variable, we find that 
the expected number $\E P_\ell(\giin)$ of paths of length $\ell$ is
  \begin{align*}
\E P_\ell(\giin) &\le
\sum_{j_0,\dots,j_\ell=1}^n
\prod_{i=1}^\ell \frac{\glk(j_{i-1},j_{i})}n
\\
&\le
\sum_{j_0,\dots,j_\ell=1}^n
n\int_{\prod_i((j_i-1)/n,j_i/n]}\prod_{i=1}^\ell \glk(x_{i-1},x_{i})
\dd x_0\dotsm\dd x_{\ell}
\\
&\le
n\int_{\sss^{\ell+1}}\prod_{i=1}^\ell \glk(x_{i-1},x_{i})
\dd x_0\dotsm\dd x_{\ell}
=n\innprod{\tlk^\ell1,1}.
  \end{align*}
Hence Lemmas \refand{Lpaths}{Ltest} imply that \eqref{tiid} holds and
$\sus(\giin)\pto\sus(\glk)$.

For $\giiin$, we observe that $\giiin$ can be seen as an induced
subgraph of $\gii{n+1}$, and thus 
\begin{equation}\label{piii}
  \E \sum_\ell P_\ell(\giiin)
\le
  \E \sum_\ell P_\ell(\gii{n+1})
\le(n+1)\sus(\glk).
\end{equation}
Hence \refL{Ltest} implies that $\sus(\giiin)\pto\sus(\glk)$.

Finally, it is easily checked that $\gin$ and $\giiin$ satisfy the
conditions of \cite[Corollary 2.12(iii)]{SJ212}, and thus are
contiguous.
Hence $\sus(\gin)\pto\sus(\glk)$ too.
(One can also compare $\gin$ and $\giin$ as in~\cite[Lemma 11]{SJ163}.)

It turns out that in probability bounds such as the one we have just proved
do not obviously transfer from $\gin$ to the original CHKNS model.
On the other hand (as we shall see below), bounds on the expected
number of paths do. Hence,
in order to analyze the original CHKNS model, we shall need
to show that
\begin{equation}\label{es_chkns}
 \limsup \E n^{-1} \sum_\ell P_\ell(\gin) \le \sus(\gl\ka).
\end{equation}
If $\gl>1/4$, then $\gl\ka$ supercritical, so $\sus(\gl\ka)=\infty$ and there is nothing to prove.
Suppose then that $\gl\le 1/4$.
We may regard $\gin$ with the vertex $1$ deleted as $\giii{n-1}$.
Writing $P(G)$ for the total number of paths in a graph $G$,
and $P^*$ for the number involving the vertex $1$,
by \eqref{piii} we thus have 
\[
 \E P(\gin)-\E P^*(\gin)= \E P(\giii{n-1}) \le n\sus(\glk),
\]
so to prove \eqref{es_chkns} it suffices to show that $\E P^*(\gin)=o(n)$.

Let $S(\gin)$ denote the number of paths in $\gin$ starting at vertex 1. 
Since a path visiting vertex 1 may be viewed as the edge disjoint
union of two paths starting there, and edges are present independently,
we have $\E P^*(\gin) \le (\E S(\gin) )^2$.
Now $\E S(\gin)$ is given by 1 plus the sum over $i$ of $1/i$ times
the expected number of paths in $\giii{n-1}$ starting at vertex $i$.
Durrett~\cite[Theorem 6]{Durrett:CHKNS} proved
the upper bound
\[
 \frac{3}{8}\frac{1}{\sqrt{ij}} \frac{(\log i+2)(\log n-\log j+2)}{\log n+4}
\]
on the expected number of paths between vertices $i$ and $j$ in the
graph $H$ on $[n]$ in which edges are present independently
and the probability of an edge $ij$, $i<j$, is $1/(4j)$ (a form
of Dubin's model; see the next section).
In fact, his result is stated for the probability that
a path is present, but the proof bounds the expected number of paths.
(The factor $1/\sqrt{ij}$ is omitted in~\cite[Theorem 6]{Durrett:CHKNS}; 
this is
simply a typographical error.)
This bound carries over to $\giii{n-1}$,
which we may regard as a subgraph of $H$. Multiplying by $1/i$
and summing,
a little calculation shows that this bound implies that
$\E S(\gin)=O(n^{1/2}/\log n)$ for $\gl=1/4$, and hence
for any $\gl\le 1/4$.
 From the comments above,
\eqref{es_chkns} follows, and for any $\gl>0$ we have
$\sus(\gin)\pto \sus(\glk)$.

Recall that the original CHKNS model $G_n$ has the same expected edge 
densities as $\gin$, but the mode of addition is slightly different,
with 0 or 1 edges added at each step, rather than a Poisson number;
this introduces some dependence between
edges. However, as noted in~\cite{SJ163}, the form of this
dependence is such that conditioning on a certain set
of edges being present can only reduce the probability
that another given edge is present.
Thus, any given path is at most as likely in $G_n$ as in $\gin$,
and \eqref{es_chkns} carries over to the CHKNS model.
On the other hand,
the effect of this dependence is small except for the first
few vertices, and it is easy to see that $N_k(G_n)$
has almost the same distribution as $N_k(\gin)$. 
In particular, $N_k(G_n)/n\pto \rho_k(\gl\kk)$, so
the proof of \refT{Tlower} goes though. Using \refL{LA1}
it follows that $\sus(G_n)\pto \sus(\gl\kk)$.

Turning to the supercritical case, let $M_k(G)$ denote the number
of components of a graph $G$, other than $\cc1$, that have order $k$.
We claim that, in all variants $\gin$, $\giin$, $\giiin$ or the original
CHKNS model, for fixed $\gl>\glc$
there is some sequence of events $\cE_n$ that holds whp,
and some $\eta>0$ such that
\begin{equation}\label{tail}
 n^{-1} \E ( M_k(G_n)\mid \cE_n) \le 100 e^{-\eta k^{1/5}},
\end{equation}
say, for all $n,k\ge 1$.
Suppose for the moment that \eqref{tail} holds.
Then
\[
 \E \susq(G_n\mid \cE_n) = n^{-1} \sum_{k\ge 1} k^2 \E (M_k(G_n)\mid\cE_n) \le \sum_k 100 k^2 e^{-\eta k^{1/5}} <\infty.
\]
For each fixed $k$ we have $n^{-1}\E k^2M_k(G_n) = n^{-1}\E (k\nk(G_n)-O(k)) \to k\rho_k(\glk)$.
Since $\cE_n$ holds whp and $n^{-1}k^2M_k(G_n)$ is bounded it follows that
$n^{-1} k^2 \E (M_k(G_n)\mid \cE_n)\to k\rho_k(\glk)$.
Hence, by dominated convergence,
$\E (\susq(G_n)\mid\cE_n) \to \sum k\rho_k(\glk) = \susq(\glk)$, and (which we know already in this case),
$\susq(\glk)$ is finite.
By \refL{Ltestcond}(ii),
it then follows that $\susq(G_n)\pto\susq(\glk)$.

To prove \eqref{tail} we use an idea from \cite{SJ163}; with an eye to
the next subsection, in the proof we shall not rely on the exact values
of the edge probabilities, only on certain bounds.
Fix $\gl>\glc$. Choosing $\eta$ small, in proving \eqref{tail} we may
and shall assume that $k$ is at least some constant that may depend on $\gl$.
Set $\delta=k^{-1/100}$,
and let $G_n'$ be the subgraph of $G_n$ induced
by the first $n'=(1-\delta)n$ vertices. (We ignore the irrelevant rounding
to integers.)
In all variants $\gin$, $\giin$, $\giiin$,
the distribution of $G_n'$ stochastically dominates that of $G_{n'}$,
so whp $G_n'$ contains a component $C$ of order at least $3\rho(\glk)n'/4\ge \rho(\glk)n/2$.
Let us condition on $G_n'$, assuming that this holds. Note that whp the largest
component of $G_n$ will contain $C$, so it suffices to bound the expectation
of $M_k'$, the number of $k$-vertex components of $G_n$ not containing $C$.
To adapt what follows to the original CHKNS model, we should instead condition
on the edges added by time $n'$ as the graph grows; we omit the details.

Suppose that $C'$ is a component of $G_n'$ other than $C$.
Consider some vertex $v$, $n'< v\le (1-\delta/2)n$.
Then $v$ has probability at least $\gl(1/v-1/n)\ge \gl\delta/(2n)\ge \delta/(8n)$
of sending
an edge to any given vertex, and hence probability at least $\delta|S|/(9n)$
of sending at least one edge to any given set $S$ of vertices.
Hence with probability at least $\delta^2\rho(\glk)|C'|/(200n)$, $v$
sends an edge to both $C$ and $C'$. Since these events are independent for
different $v$, the probability that $C'$ is not part of the same component
of $G_n$ as $C$ is at most
\[
 \bigpar{1-\delta^2\rho(\glk)|C'|/(200n)}^{\delta n/2} \le \exp\bigpar{-\delta^3\rho(\glk)|C'|/400}
 = \exp(-a\delta^3|C'|),
\]
for some $a>0$ independent of $k$.

Let $A$ be the number of components of $G_n'$ of size at least $k^{1/4}$
that are not joined to $C$ in $G_n$. Then it follows that $\E A\le ne^{-a k^{1/5}}$.

For any $v\le n'$, the expected number of edges from `late' vertices $w> n'$
to $v$ is at most $1/2$, say. (We may assume $\delta$ is small if $\gl$ is large.)
Let $B$ be the number of vertices receiving at least $k^{1/4}$ edges from late
vertices. Then it is easy to check
(using a Chernoff bound or directly)
that $\E B\le ne^{-bk^{1/4}}$ for some $b>0$.
The subgraph of $G_n$ induced by the late vertices is dominated by an
Erd\H os--R\'enyi random graph with average degree at most $1/2$.
Let $N$ be the number of components of this subgraph with size at least
$k^{1/4}$. Then, since the component exploration process is dominated by
a subcritical branching process, we have
$\E N\le ne^{-ck^{1/4}}$ for some $c>0$.

Let $M_k''$ be the number of $k$-vertex components of $G_n$ other than that
containing $C$ that do not contain any of the components/vertices
counted by $A$, $B$ or $N$. Since $\E(M_k'-M_k'') \le \E(A+B+N) \le ne^{-dk^{1/5}}$ for some $d>0$,
it suffices to bound $\E M_k''$. Condition on $G_n'$ and explore from some vertex not in $C$.
To uncover a component counted by $M_k''$, this exploration must cross from late to early vertices
at least $k^{1/4}$ times -- each time we reach a component of size at most $k^{1/4}$,
and from each of these vertices we get back to at most $k^{1/4}$ late vertices,
and from each of those to at most $k^{1/4}$ other late vertices before
we next cross over to early vertices.
However, every time we find an edge from a late to an early vertex
(conditioning on the presence of such an edge but not its destination early vertex), we have
probability at least $\rho(\glk)/2$ of hitting $C$. It follows that
$\E M_k''\le n (1-\rho(\glk)/2)^{k^{1/4}}$, and \eqref{tail} follows.

Note that since $\susq(\glk)$ is a discontinuous function of $\gl$,
we cannot obtain convergence to $\susq(\glk)$
for an arbitrary sequence $\gl_n\to\gl$, as in \refT{Tbounded} and
\refSS{SSER}. In fact, it follows easily from \refT{Tlower} that if
$\gl_n\downto \glc$ slowly enough, then
$\sus(\gnxx{\gl_n\kk})\pto\infty>\sus(\glc\kk)$ and 
$\susq(\gnxx{\gl_n\kk})>\lim_{\gl\downto\glc}\susq(\glk)-\eps
=4-\eps>\susq(\glc\kk)$ 
\whp{} for every $\eps\in(0,2)$,
for any vertex space $\vxs$ (with $\sss$ and $\mu$ as above), and thus
in particular for $\giin$.

\subsection{Dubins' model}\label{SSDubins} 

A random graph closely related to the CHKNS model is 
the graph $\gnxx{\glk}$ with kernel
\begin{align}
\label{dubins}
\kk(x,y)
\=\frac1{x\bmax y}
\end{align}
on $\sss=(0,1]$, 
where the vertex space $\vxs$ is as in \refSS{SSCHKNS}, so
$\xs=(x_1,\ldots,x_n)$.
In this case, the probability $\pij$ of an edge between $i$ and $j$ is
given (for $\gl\le1$) by
$\pij=\glk(i/n,j/n)/n =\gl/(i\bmax j)$. Note that this is independent
of $n$, so we may regard $\gnxx{\glk}$ as an induced subgraph of an
infinite random graph with vertex set $\bbN$ and these edge
probabilities, with independent edges.

This infinite random graph was introduced
by Dubins, who asked when it is a.s.\
connected. 
Shepp \cite{Shepp} 
proved that this holds if and only if $\gl>1/4$.
The finite random graph $\gnxx\glk$ 
was studied by \citet{Durrett:CHKNS,Durrett}, 
who showed that  $\glc=1/4$; thus the 
critical value for the emergence of a giant component in the
finite version coincides with the critical value for
connectedness of the infinite version.
See also \cite{SJ163,Rsmall,kernels}.

We  have
\begin{equation}
  \label{tkdubins}
\tlk f(x)
=
\frac\gl x\int_0^x f(y)\dd y 
+ \gl\int_x^1\frac{ f(y)}y\dd y.
\end{equation}

We can solve \eqref{em} as in \refSS{SSCHKNS}; we get the same
equation \eqref{julie} and thus the same solutions 
$f(x)=C_+x^{\ga_+-1} + C_- x^{\ga_--1}$ (unless $\gl=1/4$ when we also
get a logarithmic term), and substitution into \eqref{tkdubins} shows that
this is a solution of \eqref{em} if and only if $C_+\ga_+
+C_-\ga_-=1$, see \eqref{dxa} below.
If $0<\gl<1/4$, so $\ga_+>1/2>\ga_-$,
there is thus a positive solution
$f(x)=\ga_+\qw x^{\ga_+-1}$ in $L^2$.
(This is the unique solution in $L^2$, by a direct check or by \refC{Ceq}.)
Hence, \refC{Ceq} yields
\begin{equation}
  \sus(\glk)=\intoi f(x)\dd x=\ga_+\qww 
=\frac{1-2\gl-\sqrt{1-4\gl}}{2\gl^2},
\qquad 0<\gl<1/4.
\end{equation}
Since this function is analytic on $(0,1/4)$ but
has a singularity at $\gl=1/4$ (although it remains finite there), 
\refT{TB} shows that $\glc=1/4$, which gives a new proof of this
result by \citet{Durrett:CHKNS}.
Note that $\sus(\glc\kk)=4$ is finite.

We can estimate the expected number of paths as in \refSS{SSCHKNS}, and
show by Lemmas \refand{Lpaths}{Ltest} that 
$\sus(\gnxx\glk)\pto\sus(\glk)$ for any $\gl>0$.

In the supercritical case, the tail bound \eqref{tail} goes through, showing 
that for any $\gl>\glc$ we have $\susq(\glk)<\infty$, and $\susq(\gnxx\glk)\pto \susq(\glk)$.
Unfortunately, while the argument gives a tail bound on the sum $\sum_k k\rho_k(\glk)$
for each fixed $\gl>\glc$, the dependence on $\gl$ is rather bad, so it does
not seem to tell us anything about the behaviour of $\susq(\glk)$ as $\gl$
approaches the critical point.

We can easily calculate $\rho_k$ for small $k$.
First, by \eqref{tkdubins}, $\tlk1(x)=\gl-\gl\log x$.
Hence \eqref{rho1} yields
\begin{equation}
  \label{drho1}
\rhox1(\glk;x)=e^{-\gl+\gl\log x} = e^{-\gl}x^\gl.
\end{equation}
Further, instead of \eqref{sjw} we now have, for every non-zero
$\gam>-1$,
\begin{equation}\label{dxa}
  \tlk (x^\gam) =
\frac{\gl}{\gam}-\frac{\gl}{\gam(\gam+1)}x^\gam.
\end{equation}
Hence \eqref{rho2} yields
\begin{equation}
  \label{drho2}
\rhox2(\glk;x)=e^{-\gl} x^\gl\tlk(e^{-\gl}x^\gl)
=e^{-2\gl} x^\gl\Bigpar{1-\frac{x^\gl}{\gl+1}}.
\end{equation}
Similarly, by \eqref{rho3} and some calculations,
\begin{multline*}
  \label{drho3}
\rhox3(\glk;x)
=\frac{e^{-3\gl}}{2(1+\gl)^2(1+2\gl)}\\
\Bigpar{(2+3\gl)x^{3\gl}-4(1+2\gl)(1+\gl)x^{2\gl}+(2+3\gl)(1+2\gl)(1+\gl)x^\gl}
,\end{multline*}
and so on.
By integration we then obtain
\begin{align}
  \rhox1(\glk)&=\frac{e^{-\gl}}{1+\gl},
\\
  \rhox2(\glk)&=\frac{2\gl e^{-2\gl}}{(1+\gl)(1+2\gl)},
\\
  \rhox3(\glk)&=\frac{(15\gl^2+18\gl^3)e^{-3\gl}}
{2(1+\gl)^2(1+2\gl)(1+3\gl)}.
\end{align}

It is clear that each $\rho_k(\glk)$ is $e^{-k\gl}$ times a rational
funtion of $\gl$, but we do not know any general formula or a
recursion that enables us to calculate $\susq(\glk)$ in the
supercritical case as in \refSS{SSCHKNS}.

\subsection{Functions of $\max\set{x,y}$}\label{SSmax}

The examples in Sections \refand{SSCHKNS}{SSDubins} are both of the
type $\kk(x,y)=\phi(x\bmax y)$ for some function $\phi$ on $(0,1]$.
It is known that if, for example, $\phi(x)=O(1/x)$,
then $\tk$ is bounded on $L^2$, and thus there exists a
positive $\glc>0$; see 
\cite{MV,SJ139} and \cite[Section 16.6]{kernels}.

We have 
\begin{equation}
  \tlk f(x)=\gl\phi(x)\int_0^x f(y)\dd y + \gl\int_x^1\phi(y) f(y)\dd y.
\end{equation}
If $\phi\in C^1(0,1]$, then any integrable solution of
  \eqref{em} must be in $C^1(0,1]$ too, and differentiation yields
$f' = \gl \phi' F$, where $F(x)\=\int_0^x f(y)\dd y$ is the primitive
  function of $f$;
furthermore, we have $f(1)=1+\tlk f(1)=1+\gl\phi(1) F(1)$.
Hence, solving \eqref{em} is equivalent to solving the
Sturm--Liouville problem
\begin{equation}\label{sl}
  F''(x) = \gl\phi'(x) F(x)
\end{equation}
with the boundary conditions 
\begin{equation}\label{sl2}
F(0)=0\qquad \text{and}  \qquad F'(1)=\gl\phi(1) F(1)+1.
\end{equation}
If there is a solution 
to \eqref{sl} and \eqref{sl2}
with $F'\ge0$ and $F'\in L^2$, then \refC{Ceq}
shows that
\begin{equation}\label{slsus}
  \sus(\glk)=\intoi F'(x)\dd x = F(1).
\end{equation}

The examples in Sections \refand{SSCHKNS}{SSDubins} are examples of
this, as is the \ER{} case in \refSS{SSER} ($\phi=1$).
We consider one more simple explicit example.

\begin{example}
Let $\phi(x)=1-x$. Then \eqref{sl} becomes $F''=-\gl F$, with
the solution, using \eqref{sl2},  $F(x)=A\sin(\glq x)$ with 
$A\glq\cos(\glq)=1$. This solution satisfies $F'\ge0$ if $\glq<\pi/2$,
so we find $\glc=\pi^2/4$ and, by \eqref{slsus},
\begin{equation}
  \sus(\glk) = \frac{\tan(\glq)}{\glq},
\qquad \gl<\glc=\pi^2/4.
\end{equation}
\end{example}

\subsection{Further examples}

We give also a couple of counterexamples.

\begin{example}
  \label{E2}
Let $\sss=\set{1,2}$, with $\mu\set1=\mu\set2=1/2$, and let
$\kke(1,1)=2$, $\kke(2,2)=1$ and $\kke(1,2)=\kke(2,1)=\eps$ for
$\eps\ge0$.

For $\eps=0$, $\kk_0$ is reducible; given the numbers $n_1$ and $n_2$
of vertices of the two types, the random graph $\gnxx{\gl\kko}$
consists of
two disjoint independent random graphs $G(n_1,2\gl/n)$ and
$G(n_2,\gl/n)$; since $n_1/n,n_2/n\pto 1/2$, the first part has a threshold
at $\gl=1$ and the second a threshold at $\gl=2$. Similarly, the
branching process $\bpxx{\gl\kko}$ is a single-type Galton--Watson process
with offspring distribution $\Po(\gl)$ if $x=1$ and $\Po(\gl/2)$ if
$x=2$, so $\bpx{\gl\kko}$ is a mixture of these. Hence, if
$\susqer(\gl)$ denotes 
the (modified) susceptibility in the \ER{} case, given by
\eqref{suser} for $\gl<1$ and \eqref{susqer} for $\gl\ge1$, then
\begin{equation}
 \susq(\gl\kko)=\tfrac12 \susqer(\gl)+\tfrac12\susqer(\gl/2),
\end{equation}
so $\susq(\gl\kko)$ has two singularities, at $\gl=1$ and $\gl=2$.
Clearly, $\glc=1$.

Now consider $\eps>0$ and let $\eps\downto0$. 
Then
$\glc(\kke)\le\glc(\kko)=1$.
Furthermore, 
for any fixed $\gl$,
$\rho(\gl\kke,x)\to\rho(\gl\kko,x)$ 
by \cite[Theorem 6.4(ii)]{kernels}, and hence 
$\Txq{\gl\kke}\to\Txq{\gl\kko}$ (we may regard the operators as
$2\times2$ matrices). Consequently, if $\gl>1$ with $\gl\neq2$ and thus
$\norm{\Txq{\gl\kko}}<1$, then
$(I-\Txq{\gl\kke})\qw\to(I-\Txq{\gl\kko})\qw$, and thus
$\susq(\gl\kke)\to\susq(\gl\kko)$ by \refT{TBP1}. 
This holds for $\gl=2$ also, with the limit $\susq(2\kko)=\infty$, for
example by \eqref{bp1d} and Fatou's lemma.

Since $\susq(\gl\kko)$ has singularities both at 1 and 2, we may
choose $\gd\in(0,1/2)$
such that $\susq((1+\gd)\kko)>\susq(\frac32\kko)$ and
$\susq((2-\gd)\kko)>\susq(\frac32\kko)$, 
and then choose $\eps>0$ such that
$\susq((1+\gd)\kke)>\susq(\frac32\kke)$ and
$\susq((2-\gd)\kke)>\susq(\frac32\kke)$. This yields an example of an
irreducible kernel $\kk$ such that $\susq(\glk)$ is not monotone
decreasing on $(\glc,\infty)$.
\end{example}

\begin{example}
  \label{Ebad}
\refT{Tbounded} shows convergence of $\sus(\gnkx)$ to $\sus(\kk)$ for
any vertex space $\vxs$ when $\kk$ is bounded. For unbounded $\kk$, 
some restriction on the vertex space is necessary. (Cf.\
\refT{Tiid} with a very strong condition on $\vxs$ and none on $\kk$.) 
The reason is that our
conditions on $\vxs$ are weak and do not notice sets of vertices of
order $o(n)$, but such sets can mess up $\sus$.

In fact, assume that $\kk$ is unbounded.
For each $n\ge16$, find $(a_n,b_n)\in\sss^2$ with $\kk(a_n,b_n)>n$.
Define $\xs$ by taking $\floor{n^{3/4}}$ points $x_i=a_n$, 
$\floor{n^{3/4}}$ points $x_i=b_n$, and the remaining
$n-2\floor{n^{3/4}}$ points  \iid{} 
at random with distribution $\mu$.
It is easily seen that this yields a vertex space $\vxs$, and that we have
created a component with at least $2\floor{n^{3/4}}$ vertices.
Consequently, $\ccc1>n^{3/4}$, and by \eqref{sus}, 
$\sus(\gnkx)\ge\ccc1^2/n>n\qq$,
so $\sus(\gnkx)\to\infty$, even if $\kk$ is subcritical and thus
$\sus(\kk)<\infty$.

Using a similar construction (but this time for more specific kernels $\kk$),
it is easy to give examples of unbounded supercritical $\kk$
with $\susq(\kk)<\infty$ but $\susq(\gnkx)\to\infty$ for suitable vertex spaces
$\vxs$.
\end{example}

\newcommand\AAP{\emph{Adv. Appl. Probab.} }
\newcommand\JAP{\emph{J. Appl. Probab.} }
\newcommand\JAMS{\emph{J. \AMS} }
\newcommand\MAMS{\emph{Memoirs \AMS} }
\newcommand\PAMS{\emph{Proc. \AMS} }
\newcommand\TAMS{\emph{Trans. \AMS} }
\newcommand\AnnMS{\emph{Ann. Math. Statist.} }
\newcommand\AnnPr{\emph{Ann. Probab.} }
\newcommand\CPC{\emph{Combin. Probab. Comput.} }
\newcommand\JMAA{\emph{J. Math. Anal. Appl.} }
\newcommand\RSA{\emph{Random Struct. Alg.} }
\newcommand\ZW{\emph{Z. Wahrsch. Verw. Gebiete} }
\newcommand\DMTCS{\jour{Discr. Math. Theor. Comput. Sci.} }

\newcommand\AMS{Amer. Math. Soc.}
\newcommand\Springer{Springer-Verlag}
\newcommand\Wiley{Wiley}

\newcommand\vol{\textbf}
\newcommand\jour{\emph}
\newcommand\book{\emph}
\newcommand\inbook{\emph}
\def\no#1#2,{\unskip#2, no. #1,} %(typeset after year) 
\newcommand\toappear{\unskip, to appear}

\newcommand\webcite[1]{%\hfil  %???
   %\penalty0 %???
\texttt{\def~{{\tiny$\sim$}}#1}\hfill\hfill}
\newcommand\webcitesvante{\webcite{http://www.math.uu.se/~svante/papers/}}
\newcommand\arxiv[1]{\webcite{arXiv:#1.}}

\def\nobibitem#1\par{}


\begin{thebibliography}{99}

\bibitem{SJ139}
A.B. Aleksandrov, S. Janson, V.V. Peller \& R. Rochberg,
An interesting class of operators with unusual Schatten--von Neumann
behavior.
\inbook{Function Spaces, Interpolation Theory and Related Topics
(Proceedings of the International Conference in honour of Jaak Peetre
on his 65th birthday, Lund 2000)},
eds. M. Cwikel, M. Englis, A. Kufner, L.-E. Persson \& G. Sparr,
Walter de Gruyter, Berlin,
2002, pp. 61--150.

\bibitem{QRperc}
 B.~Bollob\'as, C.~Borgs, J.~Chayes and O.~Riordan,
 Percolation on dense graph sequences.
 {\em Annals of Probability}, to appear.
 \arxiv{0701346}

\bibitem[Bollob\'as, Janson and Riordan(2005)]{SJ163}
B.~Bollob\'as, S. Janson and O.~Riordan,
The phase transition in the uniformly grown random graph has infinite
order.
\RSA \vol{26} (2005), % no. 1-2, 
1-36.

\bibitem[Bollob\'as, Janson and Riordan(2007)]{kernels}
B.~Bollob\'as, S. Janson and O.~Riordan,
The phase transition in inhomogeneous random graphs.
\RSA \vol{31} (2007), 3--122.

\bibitem[Bollob\'as, Janson and Riordan(2008+)]{clustering} 
B.~Bollob\'as, S.~Janson and O.~Riordan,
 Sparse random graphs with clustering.
 Preprint (2008).
\arxiv{0807:2040}

\bibitem[Bollob\'as, Janson and Riordan(2008+)]{cutsub} 
B.~Bollob\'as, S.~Janson and O.~Riordan,
 The cut metric, random graphs, and branching processes.
 Preprint (2009).
\arxiv{0901:2091}

\bibitem[Bollob\'as and Riordan(2006)]{BRbook} 
B.~Bollob\'as \and O.~Riordan,
\book{Percolation}.
Cambridge University Press, Cambridge, 2006, x + 323~pp.


\bibitem[Bollob\'as and Riordan(2009+)]{BRmetrics}
B.~Bollob\'as and O.~Riordan,
Metrics for sparse graphs.
Preprint (2007).
\arxiv{0708.1919}


\bibitem[Borel(1942)]{Borel}
\'E. Borel, 
Sur l'emploi du th\'eor\`eme de Bernoulli pour faciliter le calcul
d'une infinit\'e de coefficients. 
Application au probl\`eme de l'attente \`a un guichet.  
\emph{C. R. Acad. Sci. Paris} \vol{214} (1942), 452--456.

\bibitem{BCLSV:1} C.~Borgs, J.T.~Chayes, L.~Lov\'asz, V.T.~S\'os and
K.~Vesztergombi, 
Convergent sequences of dense graphs I: Subgraph
frequencies, metric properties and testing.
{\em Advances in Math.} {\bf 219} (2008), 1801--1851.


\bibitem{BCLSV3}
C.~Borgs, J.T.~Chayes, L.~Lov\'asz, V.T.~S\'os and K.~Vesztergombi:
Convergent sequences of dense graphs II.
Multiway cuts and statistical physics. Preprint (2007).\\
\webcite{http://research.microsoft.com/~borgs/Papers/ConRight.pdf.}


\bibitem[Callaway, Hopcroft, Kleinberg, Newman and Strogatz(2001)]{CHKNS}
D.S.~Callaway, J.E.~Hopcroft, J.M.~Kleinberg, M.E.J.~Newman \&
S.H.~Strogatz,
 Are randomly grown graphs really random?
 \emph{Phys. Rev. E} \vol{64} (2001), 041902.

\bibitem{ChayesSmith}
 L.~Chayes and E.A.~Smith,
 Layered percolation on the complete graph.
 Preprint (2009).
\webcite{http://www.math.ucla.edu/~lchayes/}

\bibitem{DMS-anomalous}
S.N.~Dorogovtsev, J.F.F.~Mendes \& A.N.~Samukhin,
 Anomalous percolation properties of growing networks.
 \emph{Phys. Rev. E} \vol{64} (2001), 066110.

\bibitem[Dunford and Schwartz(1958)]{Dunford-Schwartz}
N. Dunford \& J.T. Schwartz,
\book{Linear Operators. I. General Theory}.
Interscience Publishers, New York, 1958.
xiv+858 pp. 

\bibitem[Durrett(2003)]{Durrett:CHKNS}
R.~Durrett,
 Rigorous result for the CHKNS random graph model.
\inbook{Proceedings, Discrete Random Walks 2003 (Paris, 2003)},
eds. C. Banderier \& Chr. Krattenthaler,
{Discrete Mathematics and Theoretical Computer Science}
 \textbf{AC} (2003),
pp. 95--104,\\
 \webcite{http://www.dmtcs.org/proceedings/html/dmAC0109.abs.html}

\bibitem[Durrett(2007)]{Durrett}
R. Durrett,
\book{Random Graph Dynamics}.
Cambridge University Press, Cambridge, 2007.%, x?? + 212~pp.

\bibitem[Dwass(1969)]{Dwass}
M. Dwass, 
The total progeny in a branching process and a related random walk.
\JAP \vol{6} (1969), 682--686.

\bibitem{FKquick} A.~Frieze and R.~Kannan,
 Quick approximation to matrices and applications.
 {\em Combinatorica} {\bf 19} (1999), 175--220.

\bibitem{SJ212}
S. Janson,
Asymptotic equivalence and contiguity of some random graphs.
Preprint (2008).
\arxiv{0802.1637}

\bibitem{SJ97}  
S. Janson, D.E. Knuth, T. {\L}uczak \& B. Pittel,
The birth of the giant component.
\RSA
\vol3 (1993),
233--358.

\bibitem[Janson and Luczak(2008+)]{SJ218}
S. Janson \& M. Luczak,
Susceptibility in subcritical random graphs.
\jour{J. Math. Phys.}, to appear.
\arxiv{0806.0252} % [math.PR]

\bibitem[Janson and Riordan(2009+)]{cutdual}
S. Janson \& O. Riordan,
Duality in inhomogeneous random graphs, and the cut metric.
Preprint (2009).
\arxiv{0905.0434}

\bibitem[Kato(1976)]{Kato}
T. Kato, 
\book{Perturbation Theory for Linear Operators}. 
2nd ed. 
\Springer, Berlin, 1976. xxi+619 pp. 

\bibitem{MV}
V.G. Maz'ya \& I.E. Verbitsky,
The Schr\"odinger operator on the energy space: boundedness and
compactness criteria.
\jour{Acta Math.}  \vol{188}  (2002),  %no. 2,
263--302.

\bibitem{McD} C.~McDiarmid,
 On the method of bounded differences.
\inbook{Surveys in Combinatorics, 1989}.
 LMS Lecture Note Series {\bf 141}, Cambridge Univ. Press (1989),  148--188.

\bibitem[Otter(1949)]{Otter} 
R. Otter, 
The multiplicative process. 
\AnnMS \vol{20} (1949), 206--224.

\bibitem[Pitman(1998)]{Pitman:enum}  
J. Pitman, 
Enumerations of trees and forests related to branching processes and
random walks. 
\emph{Microsurveys in Discrete Probability (Princeton, NJ, 1997)}, 163--180,
DIMACS Ser. Discrete Math. Theoret. Comput. Sci., 41, 
\AMS, Providence, RI, 1998. 

\bibitem[Riordan(2005)]{Rsmall}
O.~Riordan,
The small giant component in scale-free random graphs,
\CPC \vol{14} (2005), 897--938.

\bibitem{Shepp}  L.A.~Shepp,
 Connectedness of certain random graphs.
 \emph{Israel J. Math.} \vol{67} (1989), 23--33.

\bibitem[Spencer and Wormald(2007)]{SW}
J. Spencer \& N. Wormald,
Birth control for giants.  
\emph{Combinatorica}  \vol{27}  (2007),  587--628. 

\bibitem[\Takacs(1989)]{Takacs:ballots}
L. \Takacs,
Ballots, queues and random graphs. 
\JAP \vol{26} (1989),  103--112. 

\bibitem[Tanner(1961)]{Tanner}
J.C. Tanner, 
A derivation of the Borel distribution. 
\emph{Biometrika} \vol{48} (1961), 222--224.


\end{thebibliography}
\end{document}